\documentclass[14pt]{article}
\usepackage{mathpazo}
\usepackage{amsmath}
\usepackage{amsfonts}
\usepackage{array,color}
\usepackage{abstract}
\usepackage[bottom]{footmisc}
\usepackage[left=1.15in, right=1.15in, bottom=0.60in, includefoot]{geometry}
\usepackage{graphicx}
\usepackage{indentfirst}
\usepackage{amsthm}
\usepackage{mathrsfs}
\usepackage{makeidx}
\usepackage{url}
\usepackage{dsfont}
\usepackage{latexsym}
\usepackage{amsmath}
\usepackage{amssymb}
\usepackage{color}
\usepackage{mathtools}
\usepackage{hyperref}
\usepackage{tikz}
\usepackage{MnSymbol}
\numberwithin{equation}{subsection}

\newcommand{\G}{\Gamma}
\newcommand{\F}{\mathcal{F}(\mathcal M)}
\newcommand{\g}{\gamma}
\newcommand{\sg}{\sigma}
\newcommand{\mc}{\mathbb{C}}
\newcommand{\ca}{\curvearrowright}

\newcommand{\emm}{\mathcal{M}}

\newcommand{\enn}{\mathcal N}
\newcommand{\euu}{\mathcal{U}}

\newcommand{\El}{\mathcal{L}}

\newcommand{\rar}{\rightarrow}
\newcommand{\La}{\Lambda}

\newcommand{\Aa}{\mathcal{A}}
\newcommand{\bee}{\mathcal{B}}

\newcommand{\mg}{\mathcal G}
\newcommand{\mtil}{\tilde {\mathcal M}}
\newcommand{\mr}{\mathcal{R}}
\newcommand{\mrt}{\tilde{\mathcal{R}}}
\newcommand{\pee}{\mathcal{P}}

\newcommand{\id}{\operatorname{id}}
\newcommand{\Diag}{\operatorname{Diag}}
\newcommand{\Sp}{\operatorname{Sp}}
\newcommand{\GL}{\operatorname{GL}}
\newcommand{\T}{\operatorname{(T)}}

\begin{document}
	\newtheorem{Lemma}{Lemma}
	\theoremstyle{plain}
	\newtheorem{theorem}{Theorem~}[section]
	\newtheorem{main}{Main Theorem~}
	\newtheorem{lemma}[theorem]{Lemma~}
	\newtheorem{assumption}[theorem]{Assumption~}
	\newtheorem{proposition}[theorem]{Proposition~}
	\newtheorem{corollary}[theorem]{Corollary~}
	\newtheorem{definition}[theorem]{Definition~}
	\newtheorem{defi}[theorem]{Definition~}
	\newtheorem{notation}[theorem]{Notation~}
	\newtheorem{example}[theorem]{Example~}
	\newtheorem*{remark}{Remark~}
	\newtheorem{cor}{Corollary~}
	\newtheorem*{question}{Question}
	\newtheorem{claim}{Claim}
	\newtheorem*{conjecture}{Conjecture~}
	\newtheorem*{fact}{Fact~}
	\newtheorem*{thma}{Theorem A}
	\newtheorem*{corc}{Corollary C}
	\newtheorem*{thmb}{Theorem B}
	\renewcommand{\proofname}{\bf Proof}
	\newcommand{\email}{Email: }

	\title{Examples of property (T) II$_1$ factors with trivial fundamental group}
	\author{\stepcounter{footnote}Ionut Chifan\thanks{I.C. has been supported in part by the NSF grants DMS-1600688, FRG DMS-1854194 and a CDA award from the University of Iowa},\hspace{0.07in} Sayan Das, \stepcounter{footnote}Cyril Houdayer\thanks{CH is supported by ERC Starting Grant GAN 637601, Institut Universitaire de France, and FY2019 JSPS Invitational Fellowship for Research in Japan (long term)}\hspace{0.07in} and Krishnendu Khan}
	\date{}
	\maketitle
	\begin{abstract}
		In this article we provide the first examples of property (T) $\rm II_1$ factors $\mathcal N$ with trivial fundamental group, $\mathcal F (\mathcal N)=1$. Our examples arise as group factors $\mathcal N=\mathcal L(G)$ where $G$ belong  to two distinct families of property (T) groups previously studied in the literature: the groups introduced by Valette in \cite{Va04} and the ones introduced recently in \cite{CDK19} using the Belegradek-Osin Rips construction from \cite{BO06}. In particular, our results provide a continuum of explicit pairwise non-isomorphic property (T) factors.   
	\end{abstract}
	
	\section{Introduction}

	Motivated by their continuous dimension theory, Murray and von Neumann introduced the notion of $t$-by-$t$ matrix over a II$_1$ factor $\emm$, for any positive real number $t>0$, \cite{MvN43}. This is a II$_1$ factor denoted by $\emm^t$ and called the $t$-amplification of $\emm$. When $ t\leq 1$ this the isomorphism class of $p\emm p$ for a projection $p\in \emm$ of trace $\tau(p)=t$ and when $1<t$ it its the isomorphism class of $p (M_n(\mathbb C) \otimes \emm)p$ for an integer $n$ with $t/n \leq 1$ and a projection $p\in M_n(\mathbb C) \otimes \emm$ of trace $(Tr_n\otimes \tau)(p)=t/n$. One can see that up to isomorphism the $\emm^t$ does not depend on $n$ or $p$ but only on the value of $t$.  
	
	The fundamental group, $\mathcal F(\emm)$, of a II$_1$ factor $\emm$ is the set of all $t>0$ such that $\emm^t\cong \emm$. Since for any $s,t>0$ we have $(\emm^s)^t\cong \emm^{st}$ then one can see $\mathcal F(\emm)$ forms a subgroup of $\mathbb R_+$. As the fundamental group is an isomorphism invariant of the factor, its study is of central importance to the theory of von Neumann algebras. In \cite{MvN43} Murray and von Neumann were able to show  that the fundamental group of the hyperfinite $\rm II_1$ factor $\mr$ satisfies $\mathcal F(\mr)= \mathbb R_{+}$. This also implies that $\F = \mathbb R_{+} $ for all McDuff factors $\emm$. However, besides this case no other calculations were available for an extended period of time and Murray-von Neumann's original question whether $\F$ could be different from $\mathbb R_{+}$ for some factor $\emm$ remained wide open (see \cite[page 742]{MvN43} and the discussions in \cite{Po20}).  
	\vskip 0.09in
	
	A breakthrough in this direction emerged from Connes' discovery in \cite{Co80} that the fundamental group of a  group factor $\mathcal F(\El(G))$ reflects rigidity aspects of the underlying group $G$, being countable whenever $G$ has property (T) of Kazdhan \cite{Kaz67}. This finding also motivated him to formulate his famous Rigidity Conjecture in \cite{Co82} along with other problems on computing symmetries of property (T) factors---that were highlighted in subsequent articles by other prolific mathematicians \cite[Problem 2, page 551]{Co94}, \cite[Problems 8-9]{Jo00} and \cite[page 9]{Po13}. Further explorations of Connes' idea in \cite{Po86, GN87,GG88,Po95} unveiled new examples of separable factors $\emm$ with  countable $\mathcal F(\emm)$, including examples for which $\mathcal F (\emm)$ contains prescribed countable sets. However despite these advances concrete calculations of fundamental groups remained elusive for more than two decades.  
	\vskip 0.09in
	The situation changed radically with the emergence of Popa's deformation/rigidity theory in early $2000$. Through this novel theory we have witnessed an  unprecedented progress towards complete calculations of fundamental groups. The first successes in this direction were achieved by Popa and include a series of striking results: examples of factors with trivial fundamental group \cite{Po01} which answers a long-standing open problem of Kadison \cite{K67} (see \cite[Problem 3]{Ge03}); examples of factors that have \emph{any} prescribed countable subgroup of $\mathbb R_+$ as a fundamental group \cite{Po03}.  An array of other powerful results on computations of fundamental groups were obtained subsequently \cite{IPP05,PV06,Io06,Va07,PV08,Ho09, IPV10, BV12}. Remarkably, in \cite{PV08} it was shown that many uncountable proper subgroups of $\mathbb R_{+}$ can be realized as fundamental groups of separable II$_1$ factors. 
	
	\vskip 0.09in
	However, despite  these impressive achievements, significantly less is known about the fundamental groups of property (T) factors as the prior results do not apply to these factors. In fact there is no explicit calculation of the fundamental group of any property (T) factor  available in the current literature. In this article we make progress on this problem by providing two independent classes of examples of property (T) icc groups $G$ whose factors $\El(G)$ have trivial fundamental group. In particular our results advance \cite[Problem 2, page 551]{Co94} and provide the first group examples satisfying the last conjecture on page 9  in Popa's list of open problems \cite{Po13}.

	\vskip 0.09in	
	Our first class of groups $G$ is based on a minor modification of a construction introduced by Valette in \cite{Va04}. For reader's convenience  briefly describe this construction. Denote by $\mathbb H$ the division algebra of quaternions and by $\mathbb H_{\mathbb Z}$ its lattice of integer points. Fix $n \geq 2$ and recall that $\Lambda_n = \Sp(n, 1)_{\mathbb Z}$ is a lattice in the  rank one connected simple real Lie group $\Sp(n, 1)$, \cite{BHC61}. Notice that $\Sp(n, 1)$ acts linearly on $\mathbb H^{n + 1} \cong \mathbb R^{4(n + 1)}$ in such a way that $\Lambda_n$ preserves $(\mathbb H_{\mathbb Z})^{n+1} \cong \mathbb Z^{4(n + 1)}$. Then the natural semidirect product $G_n =  \mathbb Z^{4(n + 1)} \rtimes \Lambda_n$ is an icc property (T) group. Consider $\mathscr V$ the collection of all groups of the form $G= G_{n_1}\times ...\times G_{n_k}$, where $n_i\geq 2$ and $k\in \mathbb N$.    
	Combining Gaboriau's $\ell^2$-Betti numbers invariants in orbit equivalence \cite{Ga02} with the powerful uniqueness of Cartan subalgebra results of Popa and Vaes \cite{PV12} we show the groups in $\mathscr V$ give rise to an infinite family of property (T) group factors with trivial fundamental group. In fact our proof relies on the same strategy developed Popa and Vaes in their seminal work  \cite{PV11} to show that  $\mathcal F(L^\infty(X)\rtimes \mathbb F_n)=1$. 
	
	\begin{thma} \label{mainthm1}
		The following properties hold: 
		\begin{itemize}
			\item [$(\rm i)$]  For every $G\in \mathscr V$ the fundamental group satisfies $\mathcal F(\El(G))=\{1\}$;
			\item [$(\rm ii)$] The family $\{\El(G) \,:\,G \in \mathscr V\}$ consists of pairwise stably non-isomorphic II$_1$ factors.
		\end{itemize}
	\end{thma}
	
	Our second class of groups $G$ were first introduced in \cite{CDK19} and rely on a Rips construction in geometric group theory developed by Belegradek and Osin in \cite{BO06}. For convenience we briefly recall this construction. Using Dehn filling results from \cite{Os06}, it was shown in \cite{BO06} that for every finitely generated group $Q$ one can find a property (T) group $N$ such that $Q$ embeds as a finite index subgroup of ${\rm Out}(N)$. This gives rise to an action $\sigma: Q \rar {\rm Aut}(N)$ such that the corresponding semidirect product group $N\rtimes_\sigma Q$ is hyperbolic relative to $\{Q\}$.  When $Q$ is torsion free one can pick $N$ to be torsion free as well and hence both $N$ and $N \rtimes_\sigma Q$ are icc. Moreover, when $Q$ has property (T) then $N \rtimes_\sigma Q$ has property (T). Throughout this article this semidrect product $N\rtimes_\sigma Q$ will be called the Belegradek-Osin Rips construction and denoted by $Rips(Q)$.  Our examples arise as fiber products of these Rips constructions.  Specifically, consider any two groups $N_1 \rtimes_{\sigma_1} Q, N_2 \rtimes_{\sigma_2} Q \in Rip(Q)$ and form the canonical fiber product $G = (N_1 \times  N_2) \rtimes _\sigma Q$ where $\sigma =(\sigma_1, \sigma_2)$ is the diagonal action. Notice that $G$ has property (T) and the class of all these groups will be denoted by $\mathscr S$.
	\vskip 0.09in	
	Developing a new technological interplay between methods in geometric group theory and Popa's deformation/rigidity theory which continues our prior investigations \cite{CDK19} we will show that the factors associated with groups in class $\mathscr S$ have trivial fundamental group. Specifically, using various technological outgrowths of prior methods \cite{Po03,Oz03,IPP05,Io06,IPV10,Io11,PV12,CIK13,KV15,CD19,CDK19} we are able to show the following more general statement:

	\begin{thmb} \label{mainthm2}
		Assume that $Q_1$, $Q_2$, $P_1$, $P_2$ are icc, torsion free, residually finite, hyperbolic property (T) groups. Let $Q=Q_1 \times Q_2$ and $P=P_1 \times P_2$ and consider any groups $(N_1\times N_2) \rtimes Q\in \mathscr S$  and $(M_1\times M_2) \rtimes P\in \mathscr S$. Let $p \in \mathscr P(\El(M_1\times M_2)\rtimes P)$ be a projection and let $\Theta: \El((N_1 \times N_2) \rtimes Q) \rightarrow p\El((M_1 \times M_2) \rtimes P)p$ be a $\ast$-isomorphism. 
		
		Then $p=1$ and one can find a $\ast$-isomorphism, $\Theta_i: \El(N_i) \rightarrow \El(M_i)$, a group isomorphism $\delta: Q \rightarrow P$, a multiplicative character $\eta: Q \rightarrow \mathbb T$, and a unitary $u \in \mathscr U( \El((M_1 \times M_2) \rtimes P))$ such that  for all $ \g \in Q$, $x_i \in \El(N_i)$ we have that $$ \Theta((x_1 \otimes x_2)u_{\g})= \eta(\g)u(\Theta_1(x_1) \otimes \Theta_2(x_2)v_{\delta(\g)})u^{\ast}. $$ In particular, if we denote by $G= (N_1 \times N_2) \rtimes Q$ then the fundamental group satisfies $\mathcal F(\El(G))= \{1\}$.
	\end{thmb}

	Concrete examples of countable families of pairwise non-isomorphic property (T) II$_1$ factors emerged from the prior fundamental works of Cowling-Hageerup \cite{CH89} and Ozawa-Popa \cite{OP03}. Additional examples were obtained more recently, \cite{CDK19}. Since $\F$ is countable whenever $\emm$ is a property (T) factor \cite{Co80, CJ}, it also follows there exist continuum many pairwise mutually non-isomorphic property (T) factors. But, however, to the best of our knowledge, no explicit constructions of such families exist in the literature till date. Our main Theorems A and B canonically provides such examples.

	\begin{corc}\label{uncountfactt}
		For any $G=N\rtimes Q \in \mathscr S$ or $G= G_1\times ...\times G_n\in \mathscr V$, the set of all amplifications $\{\El(G)^t\,:\, t\in (0,\infty)\}$ consists of pairwise non-isomorphic II$_1$ factors with property (T).
	\end{corc}
	
	It is very plausible that, with the exception of countably many, the factors outlined in Corollary C do not appear as group factors. Therefore producing uncountable families of pairwise non-isomorphic property (T) group factors remains an open problem. In   \cite[Corollary 6.4]{CDK19} it was proposed a method to address this problem but it relies on constructing uncountably many icc, residually finite, torsion free, property (T) groups, \cite[Notation, Proposition 6.3]{CDK19}. While this seems possible with the current methods in geometric group theory there is no explicit work in the literature in this direction. 
	
	\section{Preliminaries}
	\subsection{Notations and Terminology}
	
	\vskip 0.05in	
	Throughout this document all von Neumann algebras are denoted by calligraphic letters e.g.\ $\mathcal A$, $\mathcal B$, $\mathcal M$, $\mathcal N$, etc. Given a von Neumann algebra $\mathcal M$ we will denote by $\mathscr U(\mathcal M)$ its unitary group, by $\mathscr P(\mathcal M)$ the set of all its nonzero projections and by $(\mathcal M)_1$ its unit ball.  Given a unital inclusion $\mathcal N\subseteq \mathcal M$ of von Neumann algebras we denote by $\mathcal N'\cap \mathcal M =\{ x\in \emm \,:\, [x, \enn]=0\}$.  We also denote by $\mathscr N_\emm(\enn)=\{ u\in \mathscr U(\emm)\,:\, u\enn u^*=\enn\}$ the normalizing group. We also denote the quasinormalizer of $\enn$ in $\emm$ by $\mathscr {QN}_{\mathcal M}(\mathcal N)$. Recall that $\mathscr {QN}_{\mathcal M}(\mathcal N)$ is the set of all $x\in\mathcal M$ for which there exist $x_1,x_2,...,x_n \in \mathcal M$ such that $\mathcal N x\subseteq \sum_i x_i \mathcal N$ and $x\mathcal N \subseteq \sum_i  \mathcal N x_i$ (see \cite[Definition 4.8]{Po99}).
	\vskip 0.03in
	All von Neumann algebras $\emm$ considered in this document will be tracial, i.e.\ endowed with a unital, faithful, normal linear functional $\tau:M\rightarrow \mathbb C$  satisfying $\tau(xy)=\tau(yx)$ for all $x,y\in \emm$. This induces a norm on $\emm$ by the formula $\|x\|_2=\tau(x^*x)^{1/2}$ for all $x\in \emm$. The $\|\cdot\|_2$-completion of $\emm$ will be denoted by $L^2(\emm)$.  For any von Neumann subalgebra $\mathcal N\subseteq \mathcal M$ we denote by $E_{\mathcal N}:\mathcal M\rightarrow \mathcal N$ the $\tau$-preserving condition expectation onto $\mathcal N$. We denote the orthogonal projection from $L^2(\emm) \rightarrow L^2(\enn)$ by $e_{\enn}$. The Jones' basic construction \cite[Section 3]{Jo83} for $\enn \subseteq \emm$ will be denoted by $\langle \emm, e_{\enn} \rangle$.  
	\vskip 0.05in
	For any group $G$ we denote by $(u_g)_{g\in G} \subset \mathscr U(\ell^2G)$ its left regular representation, i.e.\ $u_g(\delta_h ) = \delta_{gh}$ where $\delta_h:G\rightarrow \mathbb C$ is the Dirac function at $\{h\}$. The weak operatorial closure of the linear span of $\{ u_g\,:\, g\in G \}$ in $\mathscr B(\ell^2 G)$ is called the group von Neumann algebra and will be denoted by $\El(G)$; this is a II$_1$ factor precisely when $G$ has infinite non-trivial conjugacy classes (icc). If $\mathcal M$ is a tracial von Neumann algebra and $G \ca^\sigma \mathcal M$ is a trace preserving action we denote by $\mathcal M \rtimes_\sigma G$ the corresponding cross product von Neumann algebra \cite{MvN37}. For any subset $K\subseteq G$ we denote by $P_{\mathcal M K}$  the orthogonal projection from the Hilbert space $L^2(\mathcal M \rtimes G)$ onto the closed linear span of $\{x u_g \,|\, x\in \mathcal M, g\in K\}$. When $\mathcal M$ is trivial we will denote this simply by $P_K$.  
	\vskip 0.05in
	All groups considered in this article are countable and will be denoted by capital letters $A$, $B$, $G$, $H$, $Q$, $N$ ,$M$,  etc. Given groups $Q$, $N$ and an action $Q\curvearrowright^{\sg} N$ by automorphisms we denote by $N\rtimes_\sigma Q$ the corresponding semidirect product group. For any $n\in N$ we denote by ${\rm Stab}_Q(n)=\{ g\in Q\,:\, \sigma_g(n)=n\}$. Given a group inclusion $H \leqslant G$ sometimes we consider the centralizer $C_G(H)$ and the  virtual centralizer $vC_G(H)=\{g\in G \,:\, |g^{H}|<\infty\} $. We also denote by $\llangle H\rrangle$ the normal closure of $H$ in $G$.

	\subsection{Popa's Intertwining Techniques} Over more than fifteen years ago, Popa  introduced  in \cite [Theorem 2.1 and Corollary 2.3]{Po03} a powerful analytic criterion for identifying intertwiners between arbitrary subalgebras of tracial von Neumann algebras. Now this is known in the literature  as \emph{Popa's intertwining-by-bimodules technique} and has played a key role in the classification of von Neumann algebras program via Popa's deformation/rigidity theory.

	\begin {theorem}\cite{Po03} \label{corner} Let $(\mathcal M,\tau)$ be a separable tracial von Neumann algebra and let $\mathcal P, \mathcal Q\subseteq \mathcal M$ be (not necessarily unital) von Neumann subalgebras. 
	Then the following are equivalent:
	\begin{enumerate}
		\item There exist $ p\in  \mathscr P(\mathcal P), q\in  \mathscr P(\mathcal Q)$, a $\ast$-homomorphism $\theta:p \mathcal P p\rightarrow q\mathcal Q q$  and a partial isometry $0\neq v\in q \mathcal M p$ such that $\theta(x)v=vx$, for all $x\in p \mathcal P p$.
		\item For any group $\mathcal G\subset \mathscr U(\mathcal P)$ such that $\mathcal G''= \mathcal P$ there is no sequence $(u_n)_n\subset \mathcal G$ satisfying $\|E_{ \mathcal Q}(xu_ny)\|_2\rightarrow 0$, for all $x,y\in \mathcal  M$.
		\item There exist finitely many $x_i, y_i \in \mathcal M$ and $C>0$ such that  $\sum_i\|E_{ \mathcal Q}(x_i u y_i)\|^2_2\geq C$ for all $u\in \mathcal U(\mathcal P)$.
	\end{enumerate}
\end{theorem} 
\vskip 0.02in
\noindent If one of the three equivalent conditions from Theorem \ref{corner} holds then we say that \emph{ a corner of $\mathcal P$ embeds into $\mathcal Q$ inside $\mathcal M$}, and write $\mathcal P\prec_{\mathcal M}\mathcal Q$. If we moreover have that $\mathcal P p'\prec_{\mathcal M}\mathcal Q$, for any projection  $0\neq p'\in \mathcal P'\cap 1_{\mathcal P} \mathcal M 1_{\mathcal P}$ (equivalently, for any projection $0\neq p'\in\mathscr Z(\mathcal P'\cap 1_{\mathcal P}  \mathcal M 1_{P})$), then we write $\mathcal P\prec_{\mathcal M}^{s}\mathcal Q$. We refer the readers to the survey papers \cite{Po07, Va10icm, Io18icm} for recent progress in von Neumann algebras using deformation/rigidity theory.
\vskip 0.02in
We also recall the notion of relative amenability introduced by N. Ozawa and S. Popa. Let $(\emm, \tau) $ be a tracial von Neumann algebra. Let $p \in \emm$ be a projection, and let $\mathcal P \subseteq p \emm p$, and $\mathcal Q \subseteq \emm$ be von Neumann subalgebras. Following \cite[Definition 2.2]{OP07}, we say that $\mathcal P$ is amenable relative to $\mathcal Q$ inside $\emm$, if there exists a positive linear functional $\phi: p \langle \emm, e_{\mathcal Q} \rangle p \rightarrow \mc$ such that $\phi|_{p \emm p}= \tau$ and $\phi(xT)=\phi(Tx)$ for all $T \in \mathcal Q$ and all $x \in \mathcal P$. If $\mathcal P$ is amenable relative to $\mathcal Q$ inside $\emm$, we write $\mathcal P \lessdot_{\emm} \mathcal Q$.
\vskip 0.02in
For further use we record the following result which controls the intertwiners in algebras arsing form malnormal subgroups. Its proof is essentially contained in \cite[Theorem 3.1]{Po03} so it will be left to the reader. 	  
\begin{lemma}[Popa \cite{Po03}]\label{malnormalcontrol}Assume that $H\leqslant G$ be an almost malnormal subgroup and let $G \ca \mathcal N$ be a trace preserving action on a finite von Neumann algebra $\mathcal N $. Let $\mathcal P \subseteq \mathcal N \rtimes H$ be a von Neumann algebra such that $\mathcal P\nprec_{\mathcal N\rtimes H}  N$. Then for every elements $x,x_1,x_2,...,x_l \in \mathcal N\rtimes G$ satisfying $\mathcal P x\subseteq \sum^l_{i=1} x_i \mathcal P$ we must have that $x\in \mathcal N\rtimes H$. 
\end{lemma}
The following result is a mild generalization of \cite[Lemma 2.3]{BV12}. For reader's convenience we include all the details in our proof.
\begin{theorem}\label{joinintertwining}
	Let $G$ be a group together $H\lhd G$ with a normal subgroup and assume that $G\ca(\enn,\tau )$ is a trace preserving action on a von Neumann algebra $(\enn,\tau)$. Consider $\enn\rtimes G=\emm$ the corresponding crossed product von Neumann algebra, assume that $\Aa\subset \emm $ (possibly non-unital) and $\mg\subseteq \enn_{1_\Aa\emm 1_\Aa}(\Aa)$ a group of unitaries such that $\Aa,\mg"\prec^{s}_{\emm} \enn\rtimes H$. Then $(\Aa\mg)"\prec^{s}_{\emm} \enn \rtimes H$. 
\end{theorem}

\begin{proof} Let $G_H\subset G$ be a section for $G/H$. Also denote by $\pee= \enn \rtimes H$. Since $\Aa, \mathcal G''\prec^{s}_\emm \pee$, then by \cite[Lemma 2.5]{Va10}, for all $ \varepsilon_1,\varepsilon_2 >0$\ there exist $K_{\varepsilon_1},L_{\varepsilon_2}\subset G_H$ such that for all $a\in (\Aa)_1$ and $b\in (\mathcal G'')_1$ we have 1) $\|P_{\pee K_{\varepsilon_1}}(a)-a\|_2\leq \varepsilon_1$ and 2) $\|P_{\pee L_{\varepsilon_2}}(b)-b\|_2\leq \varepsilon_2$. Here for every $ S\subset G_H$, the map $P_{\pee S}: L^2(\emm)\rightarrow \overline{\rm span}^{\|\cdot\|_2}\{\pee u_g\,:\, g\in S \}$ is the orthogonal projection. Also notice that, for all $x\in \emm$, $P_{\pee S}(x)=\underset{s\in S}{\sum}E_{\pee}(xu_{s^{-1}})u_s$. In particular, for all $x\in \emm$ we have,
	\begin{align}
	\|P_{\pee S}(x)\|_{\infty}\leq |S|\|x\|_{\infty} \ \text{ and }\ \|P_{\pee S}(x)\|_2\leq \|x\|_2.
	\end{align}	
	Now for all $a\in (\Aa)_1,b\in (\mathcal G'')_1$ we have 
	\begin{align}
	\|ab-P_{\pee K_{\varepsilon_1}}(a)P_{\pee L_{\varepsilon_2}}(b)\|_2 & \leq \|ab-P_{\pee K_{\varepsilon_1}}(a)b\|_2+\|P_{\pee K_{\varepsilon_1}}(a)b-P_{\pee K_{\varepsilon_1}}(a)P_{\pee L_{\varepsilon_2}}(b)\|_2\nonumber \\
	& \leq \|a-P_{\pee K_{\varepsilon_1}}(a)\|_2\|b\|_{\infty}+\|P_{\pee K_{\varepsilon_1}}\|_{\infty}\|b-P_{\pee L_{\varepsilon_2}}(b)\|_2\\
	& \leq \|a-P_{\pee K_{\varepsilon_1}}(a)\|_2+|K_{\varepsilon_1}|\|b-P_{\pee L_{\varepsilon_2}}(b)\|_2\nonumber\\
	&\leq \varepsilon_1+|K_{\varepsilon_1}|\varepsilon_2. 
	\end{align}
	So letting $\varepsilon_1=\varepsilon$ and $\varepsilon_2=\frac{\varepsilon}{|K_{\varepsilon_1}|}$ we get that there exists $K_{\varepsilon},L_{\varepsilon}$ finite subsets of the section  $G/H$ such that
	\begin{align}\label{epsilon} 
	\|ab-P_{\pee K_{\varepsilon}}(a)P_{\pee L_{\varepsilon}}(b)\|\leq 2\varepsilon.
	\end{align} 
	Since $H\lhd G$, then there exist a finite set $F_{\varepsilon}\subseteq G_H$ such that $|F_{\varepsilon}|\leq|K_{\varepsilon}||L_{\varepsilon}|$ and $P_{\pee F_{\varepsilon}}(P_{\pee K_{\varepsilon}}(a)P_{\pee L_{\varepsilon}}(b))=P_{\pee K_{\varepsilon}}(a)P_{\pee L_{\varepsilon}}(b)$ for all $a\in \mathscr{U}(\Aa),b\in(\mathcal G'')_1$. Using this fact together with (\ref{epsilon}) we get that $\|P_{\pee F_{\varepsilon}}(ab)-P_{\pee K_{\varepsilon}}(a)P_{\pee L_{\varepsilon}}(b)\|\leq 2\varepsilon$ and combining with (\ref{epsilon}) again we get that 
	\begin{align}\label{6} 
	\|ab-P_{\pee F_{\varepsilon}}(ab)\|\leq 2\varepsilon.
	\end{align}
	for all $a\in \mathscr{U}(\Aa),b\in(\mathcal G'')_1$. Since $(\mathscr{U}(\Aa)\mg)"=(\Aa \mg)"$, this already shows that $(\Aa \mg)"\prec \pee$. Next we argue that we actually have $(\Aa \mg)"\prec^s \pee$. To see this fix $p\in (\Aa \mg)'\cap 1_{\Aa \vee \mg''}\emm 1_{\Aa \vee \mg''}$. Then there exists a finite set $G_{\varepsilon}\subseteq G_H$ such that 
	\begin{align}\label{5}
	\|p-P_{\pee G_{\varepsilon}}(p)\|\leq \frac{\varepsilon}{|K_{\varepsilon}||L_{\varepsilon|}}.
	\end{align} 	 
	Combining (\ref{5}) and (\ref{6}) we get that 
	\begin{align}\label{7}
	\|abp-P_{\pee F_{\varepsilon}}(ab)P_{\pee G_{\varepsilon}}(p)\|& \leq \|abp-P_{\pee F_{\varepsilon}}(ab)p\|+\|P_{\pee F_{\varepsilon}}(ab)p-P_{\pee F_{\varepsilon}}(ab)P_{\pee G_{\varepsilon}}(p)\|\nonumber \\
	&\leq \|ab-P_{\pee F_{\varepsilon}}(ab)\|_2\|p\|_{\infty} +\|P_{\pee F_{\varepsilon}}(ab)\|_{\infty}\|p-P_{\pee G_{\varepsilon}}(p)\|_2\nonumber\\
	& \leq 4\varepsilon +|F_{\varepsilon}|\cdot \frac{\varepsilon}{|K_{\varepsilon}||L_{\varepsilon}|}<5\varepsilon. 
	\end{align}
	Again there exists a finite set $T_{\varepsilon}\subset G$ such that $P_{\pee T_{\varepsilon}}(P_{\pee F_{\varepsilon}}(ab)P_{\pee G_{\varepsilon}}(p))=P_{\pee F_{\varepsilon}}(ab)P_{\pee G_{\varepsilon}}(p)$ and $|T_{\varepsilon}|\leq |F_{\varepsilon}||G_{\varepsilon}|$. Using this and (\ref{7}) we get that $\|abp-P_{\pee T_{\varepsilon}}(abp)\|<10\varepsilon$ for all $a\in\mathscr{U}(\Aa),b\in \mg$. This shows that $(\Aa \mg)"\prec^s_{\emm} \pee$, as desired.
\end{proof}

We end this section by highlighting a straightforward corollary of Theorem \ref{joinintertwining} that we will be very useful in the sequel.  
\begin{cor} \label{cor:join}
	Let $H\lhd G$ be a normal subgroup of $G$ and $G\ca(\enn,\tau )$ be a trace preserving action on a tracial von Neumann algebra $(\enn,\tau)$. Let $\emm = \enn \rtimes G$. Assume that $\Aa, \ \bee \subseteq \emm$ are commuting $\ast$-subalgebras such that $\Aa \prec_{\emm}^s \enn \rtimes H$ and $\bee \prec_{\emm}^s \enn \rtimes H$. Then $\Aa \vee \bee \prec^s_{\emm} \enn \rtimes H$.  
\end{cor}
\begin{proof} Follows from Theorem \ref{joinintertwining} by letting $\mathcal G= \mathscr U(\bee)$.\end{proof}

\subsection{Height of Elements in Group von Neumann Algebras}
The notion of height of elements in crossed products and group von Neumann algebras was introduced and developed in \cite{Io11} and \cite{IPV10} and was highly instrumental in many of the recent classification results in von Neumann algebras \cite{Io11,IPV10,KV15,CU18}. Following \cite[Section 3]{IPV10} for every $x\in \El(G)$ we denote by $h_{G}(x)$ the largest Fourier coefficient of $x$, i.e.,
$h_{G}(x)={\max}_{g\in G }\ |\tau(xu_g^*)|$.  Moreover,  for every subset $\mg\subseteq \El(\G)$, we denote by
$h_{G}(\mg)=\inf_{x\in \mg}h_{G}(x)$, the height of $\mg$ with respect to $G$.   Using the notion of height Ioana, Popa and Vaes proved in their seminal work, \cite[Theorem 3.1]{IPV10} that whenever $G$, $H$ are icc groups such that $\El(G) = \El(H)$ and $h_{G}(H)>0$, then $G$ and $H$ are isomorphic. The following generalization of this result to embeddings was obtained by Krogager and Vaes \cite{KV15} and will be used in an essential way to derive our main Theorem \ref{mainthm} in the last section.
\begin{theorem}[Theorem 4.1, \cite{KV15}] \label{kv15}
	Let $G$ be a countable group and denote by $\emm = \mathcal L(G)$. Let  $p\in\mathscr{P}(\emm)$ be a projection and assume that $\mg\subseteq \mathscr{U}(p\emm p)$ is a subgroup satisfying following properties:
	\begin{enumerate}
		\item The  unitary representation $\{Ad\ v\}_{v\in\mg}$ on $L^2(p(\emm p\ominus \mathbb{C}p)$ is weakly mixing; 
		\item For any $e\neq g\in\El(G)$ we have  $\mg"\nprec_\emm \El(C_{G}(g))$;
		\item We have $h_{G}(\mg)>0$. 
	\end{enumerate}
	Then $p=1$ and there exists a unitary $u\in \El(G)$ such that $u\mg u^*\subseteq \mathbb{T}G$.
\end{theorem}
Next we highlight a new situation  when it's possible to control  lower bound for height of unitary elements in the context of crossed product von Neumann algebras arising from group actions by automorphisms with no non-trivial stabilizers. Our result and its proof is reminiscent of the prior powerful techniques  for Bernoulli actions introduced in   \cite[Theorem 5.1]{IPV10} (see also \cite[Theorem 6.1]{Io11}) and their recent counterparts for the Rips constructions \cite[Theorem 5.1]{CDK19}. The precise statement is the following
\begin{theorem} \label{height}
	Let $G$ and $H$ be countable groups and let $\sigma: G\rar Aut(H)$ be an action by automorphisms for which   there exists a scalar $c>0$ satisfying $|{\rm Stab}_G(h)|<c$ for all $h\in H\setminus \{e\}$.   Consider $\emm = \El(H\rtimes_\sigma G)$ and let $\Aa \subseteq \emm$ be a diffuse von Neumann subalgebra such that $\Aa \prec_{\emm}^s \El(H)$. For any group of unitaries  $\mg \subseteq \El(G)$ satisfying $\mg \subseteq \mathcal N_{\emm}(\Aa)$ we have that $h_{G}(\mg)>0$.
\end{theorem}
\begin{proof} For ease of exposition denote by $\mathcal N =\El(H)$. Next we prove the following property 
	
	\begin{claim}\label{heightb1}	For every $x,y \in \mathcal L(G)$, every finite subsets  $K,S\subset G$, every $a\in {\rm span}\mathcal N K$ with $E_{\mathcal L(G)}(a)=0$ and every $\varepsilon>0$ there exists a scalar $\kappa_{\varepsilon,K,S,a}>0$ such that 
		\begin{align}\label{heightinequality}
		\|P_{\enn S}(xay)\|^2_2\leq \kappa_{\varepsilon ,K,S,a}\|y\|^2_2\|a\|^2_2h_G^2(x)+\varepsilon\|x\|_{\infty}\|y\|_{\infty},
		\end{align} 
		where $P_{\enn S}$ denotes the orthogonal projection from $L^2(\emm)$ onto $\overline{{\rm span}}^{\|\cdot\|_2} (\mathcal N S)$. 	
		
	\end{claim}
	
	\noindent \emph{Proof of Claim \ref{heightb1}.}
	First fix a finite set $L\subseteq H\setminus \{e\}$ and let $b\in {\rm span} (LK)$. Observe that using the Fourier decomposition of $x=\sum_g x_g u_g$ and $y=\sum y_g u_g$, where $x_g= \tau(xu_{g^{-1}})$ and $y_g= \tau(yu_{g^{-1}})$, basic calculations show that 	\begin{align}\label{heightcal}
	&\|E_{\mathcal N}(xby)\|^2_2 
	= \|\underset{g\in G,k\in K}{\sum} x_g y_{k^{-1}g^{-1}}\sigma_g(E_{\mathcal N}(bu_{k^{-1}}))\|_2^2\nonumber \\
	&=\underset{g_1,g_2\in G,k_1,k_2\in K}{\sum} x_{g_1}y_{k^{-1}_1g_1^{-1}}\overline{x_{g_2}}\overline{y_{k^{-1}_2 g_2^{-1}}}\langle \sigma_{g_1}(E_{\mathcal N}(bu_{k_1^{-1}})),\sigma_{g_2}(E_{\mathcal N}(bu_{k_2^{-1}}))\rangle.
	\end{align}
	Furthermore, using the Fourier decomposition $b= \sum_h b_h u_h$ where $b_h =\tau(bu_{h^{-1}})$ we also see that 
	\begin{align}\label{stablizer}
	\langle \sigma_{g_1}(E_{\mathcal N}(bu_{k_1^{-1}})),\sigma_{g_2}(E_{\mathcal N}(bu_{k_2^{-1}}))\rangle & 
	=\underset{l_1,l_2\in L}{\sum} b_{k_1l_1}\overline{b_{k_2l_2}} \delta_{\sigma_{g_1}(l_1),\sigma_{g_2}(l_2) }=\underset{l_1,l_2\in L, g_2^{-1}g_1\in S_{l_1,l_2}}{\sum}b_{k_1l_1}\overline{b_{k_2l_2}},  
	\end{align}
	where for every $l_1,l_2\in L$ we have denoted by $S_{l_1,l_2}=\{ g\in G\,:\, \sigma_{g}(l_1)=l_2\}$.
	
	Thus, combining (\ref{heightcal}) and (\ref{stablizer}) and using basic inequalities together with $|S_{l_1,l_2}|\leq c$ we get that 
	\begin{align}
	\|E_{\mathcal N}(xby)\|^2_2&\leq  \underset{k_1,k_2\in K; l_1,l_2\in L, g_1,g_2\in G, g_2^{-1}g_1\in S_{l_1,l_2} }{\sum}  \left |x_{g_1}y_{k^{-1}_1g_1^{-1}}b_{k_1l_1}\overline{x_{g_2}y_{k^{-1}_2 g_2^{-1}}  b_{k_2l_2}}\right |  \nonumber\\
	&\leq  \underset{k_1,k_2\in K, l_1,l_2\in L,  s\in S_{l_1,l_2}, g\in G,}{\sum} \left |x_{gs} y_{k_1^{-1}s^{-1}g^{-1}}b _{k_1l_1}\overline{ x_g y_{k^{-1}_2 g}  b_{k_2 l_2}}\right |\nonumber\\
	& \leq (\max_{l_1,l_2\in L}|S_{l_1,l_2}|) |K|^2|L|^2 h^2_G(x) \|y\|_2^2\|b\|_2^2\leq c|K|^2|L|^2 h^2_G(x) \|y\|_2^2\|b\|_2^2.\label{heightb2}
	\end{align}
	Using these estimates we are now ready to derive the proof of \eqref{heightinequality}.  
	To this end fix $\varepsilon >0.$ Using basic approximations and $\|E_{\El(G)}(a)\|=0$ one can find a finite set $L\subset H\setminus\{e\}$ and $b\in {\rm span}(LK)$ such that 
	\begin{align}\label{normestimate}
	\|a-b\|_2\leq\min \{\frac{\varepsilon}{2}, \|a\|_2\} \text{ and } \|b\|_{\infty}\leq 2\|a\|_{\infty}.
	\end{align}
	Notice  that for all $z\in\emm$ we have  $P_{\enn S}(z)=\sum E_{\enn}(zu_{s^{-1}})u_s$ and using this formula together with estimate (\ref{normestimate}) and Cauchy-Schwarz inequality we get
	\begin{align}\label{estimate}
	\|P_{\enn S}(xay)\|^2_2 & \leq 2|S|\left(\underset{s\in S}{\sum}\|E_{\enn}(xbyu_{s^{-1}}\|_2^2\right )+\varepsilon |x\|_{\infty}\|y\|_{\infty}.\nonumber
	\end{align}
	Using (\ref{heightb2}) followed by \eqref{normestimate} we further have that the last inequality above is smaller than  
	\begin{align}
	&\leq 2c|S||K|^2|L|^2(\underset{s\in S}{\sum} h^2_G(x)\|yu_{s^{-1}}\|_2^2\|b\|_2^2)+\varepsilon |x\|_{\infty}\|y\|_{\infty}\nonumber\\
	&\leq 4c|S|^2|K|^2|L|^2h^2_G(x)\|a\|_2^2\|y\|^2_2+\varepsilon |x\|_{\infty}\|y\|_{\infty}.
	\end{align}
	Combining this with (\ref{normestimate}) proves the claim where $\kappa_{\varepsilon, K,S,a}=4c|S|^2|K|^2|L|^2$. $\hfill\blacksquare$
	\vskip 0.07in
	In the remaining part we complete the proof of the statement. Towards this first notice that, since $\Aa\prec^s_\emm\enn$ then by  \cite[Lemma 2.5]{Va10} for every  $ \varepsilon$  there exists a finite set  $S\subseteq K$ such that for all $c\in\mathscr{U}(\Aa)$ we have 
	\begin{align}\label{***} 
	\|c-P_{\enn S}(c)\|_2\leq \varepsilon. 
	\end{align}

	Next we also claim that for every finite set $S\subset G$ and every $\varepsilon >0$ there exists $b\in \mathscr{U}(\Aa)$ such that
	\begin{align}\label{**} 
	\|E_{\El(G)}\circ P_{\enn S}(b)\|_2<\varepsilon.
	\end{align} 
	Indeed, to see this first notice that $\|E_{\El(G)}\circ P_{\enn S}(b)\|^2_2=\underset{s\in S}{\sum}|\tau(bu_{s^{-1}})|^2.$ As $\Aa$ is diffuse and $S$ is finite there exists $b\in \mathscr{U}(\Aa)$ such that $\underset{s\in S}{\sum}|\tau(bu_{s^{-1}})|^2<\varepsilon$ and the claim follows.
	\vskip 0.05in
	
	Now pick $b\in\mathscr U(\Aa)$ satisfying (\ref{**}). Since $a,gsg^{-1}\in \mathscr U(\Aa)$ then using \eqref{***} two times and \eqref{**} we see that 
	\begin{align}
	&1-\varepsilon =\|gag^{-1}\|_2-\varepsilon \leq \|P_{\enn S}(gag^{-1}\|_2\leq \|P_{\enn S}(g(P_{\enn S}(b))g^{-1})\|_2+\varepsilon\nonumber\\
	&\leq \|P_{\enn S}(g(P_{\enn S}(b)-E_{\El(G)}(P_{\enn S}(b)))g^{-1})\|_2+\|E_{\El(G)}(P_{\enn S}(b))\|_2+\varepsilon\nonumber \\ &\leq \|P_{\enn S}(g(P_{\enn S}(b)-E_{\El(G)}(P_{\enn S}(b)))g^{-1})\|_2  +2\varepsilon. \label{heightb4} 
	\end{align}
	Now, taking $a=P_{\enn S}(b)-E_{\El(G)}(P_{\enn S}(b))$ and using \eqref{heightinequality} we get that the last inequality above is smaller than
	\begin{align}
	\leq \kappa_{\varepsilon,S,S,b} h_G(g) \|P_{\enn S}(b)-E_{\El(G)}(P_{\enn S}(b))\|_2+\varepsilon^{1/2}+ 2\varepsilon. \label{heightb3}
	\end{align}
	Thus \eqref{heightb4} and \eqref{heightb3} further imply that  $h_G(g)\geq\kappa_{\varepsilon, S,S,b}^{-1}(1-3\varepsilon- \varepsilon ^{1/2})$. Since this holds for all $g\in \mg$, letting $\varepsilon>0$ be sufficiently small we get the desired conclusion.\end{proof}

\section{Two Distinguished Classes of Property (T) Groups}

In this section we describe two independent classes of groups with property (T). Our main results on calculation of fundamental groups apply to factors arising from these classes. The first class, denoted by $\mathscr V$, is described in subection~\ref{Val} and is based on construction by Valette \cite{Va04}. The second class, denoted by $\mathscr S$, is described in subsection ~\ref{bo} and was previously introduced in \cite{CDK19} using a Rips construction in geometric group theory developed by Belegradek-Osin \cite{BO06}. We also highlight several algebraic properties of these groups and their von Neumann algebras that will be essential to derive our main results in the sequel.

\subsection{Class $\mathscr V$} \label{Val}

We describe a construction of group pairs with property (T) developed by Valette \cite{Va04}. Denote by $\mathbb H$ the division algebra of quaternions and by $\mathbb H_{\mathbb Z}$ its lattice of integer points. Let $n \geq 2$. Recall that $\Sp(n, 1)$ is the rank one connected simple real Lie group defined by 
$$\Sp(n, 1) = \left\{ A \in \GL_{n+ 1}(\mathbb H) \mid A^* J A = J\right\}$$
where $J = \Diag(1, \dots, 1, -1)$.  Since the subgroup $\Sp(n, 1)$ is the set of real points of an algebraic $\mathbb Q$-group, the group of integer points $\Lambda_n = \Sp(n, 1)_{\mathbb Z}$ is a lattice in $\Sp(n, 1)$ by Borel--Harish-Chandra's result \cite{BHC61}. Observe that $\Sp(n, 1)$ acts linearly on $\mathbb H^{n + 1} \cong \mathbb R^{4(n + 1)}$ in such a way that $\Lambda_n$ preserves $(\mathbb H_{\mathbb Z})^{n+1} \cong \mathbb Z^{4(n + 1)}$. For every $n \geq 2$, consider the natural semidirect product $G_n =  \mathbb Z^{4(n + 1)} \rtimes \Lambda_n$. Throughout this documents we denote by $\mathscr V$ the collection of all finite direct product groups of the form $G= G_{n_1}\times ...\times G_{n_k}$, where $n_i\geq 2$ and $k\in \mathbb N$. Also for a group $G_n \in \mathscr V$, we denote by $\emm_n= \El(G_n)$, and by $\Aa_n= \mathcal L(\mathbb Z^{4(n + 1)})$. Note that $\emm_n =  \Aa_n \rtimes \Lambda_n$.

For further use we record some properties of the groups $G_n \in \mathscr V$ and their von Neumann algebras $\emm_n$.

\begin{theorem} \label{Valettegr}
	Let $G_n \in \mathscr V$ with $n\geq 2$. Then the following hold true:
	\begin{itemize}
		\item [$(\rm i)$]  $G_n$ is an infinite icc countable discrete group with property $\T$ so that $\emm_n$ is a ${\rm II_1}$ factor with property $\T$.
		\item [$(\rm ii)$]   $\Aa_n \subseteq \emm_n$ is the unique Cartan subalgebra, up to unitary conjugacy.
		
	\end{itemize}
\end{theorem}

\begin{proof}
	$(\rm i)$ We use the notation $ g = (a, \gamma) \in   \mathbb Z^{4(n + 1)} \rtimes \Lambda_n = G_n$. Since the lattice $\Lambda_n/\{\pm \id\}$ in the adjoint Lie group $\Sp(n, 1)/\{\pm \id\}$ is icc, the conjugacy class of any element of the form $g = (a, \gamma)$ in $G_n$ with $\gamma \notin \{\pm \id\}$ is infinite. The $\mathbb Z^{4(n + 1)}$-conjugacy class of any element of the form $g = (a, -\id)$ in $G_n$ is also clearly infinite. Moreover, the exact same proof as \cite[Theorem 4, Step 3]{Va04} shows that the conjugacy class of any element of the form $g = (a, \id)$ in $G_n$ with $a \neq 0$ is infinite. It follows that $G_n$ is an infinite icc countable discrete group. By \cite[Proposition 1]{Va04}, the group pair $(\mathbb R^{4(n + 1)} \rtimes \Sp(n, 1), \mathbb R^{4(n + 1)})$ has relative property (T). Since both $ \mathbb Z^{4(n + 1)} \rtimes \Lambda_n  <  \mathbb R^{4(n + 1)} \rtimes \Sp(n, 1)$ and $\mathbb Z^{4(n + 1)} < \mathbb R^{4(n + 1)}$ are lattices, the group pair $( \mathbb Z^{4(n + 1)} \rtimes \Lambda_n, \mathbb Z^{4(n + 1)})$ also has property (T). Since $\Sp(n, 1)$ has property (T) by Kostant's result, so does its lattice $\Lambda_n < \Sp(n, 1)$. Altogether, this implies that $G_n$ has property (T). Hence $\emm_n= \El(G_n)$ has property (T) by \cite{CJ}. \vskip 0.02in 
	\vskip 0.07in
	
	\noindent $(\rm ii)$ We first show that $\Aa_n \subseteq \emm_n$ is a Cartan subalgebra. Note that it suffices to show that $\Aa_n \subseteq \emm_n$ is maximal abelian. To this end, it is enough to show that the $\mathbb Z^{4(n + 1)}$-conjugacy class in $G_n$ of any element of the form $g = (0, \gamma)$ with $\gamma \neq \id$ is infinite. Indeed, if $\gamma \in \Lambda_n$ is such that the $\mathbb Z^{4(n + 1)}$-conjugacy class of $g = (0, \gamma)$ in $\Gamma_n$ is finite, since $\mathbb Z^{4(n + 1)}$ is torsion-free, this forces $\gamma$ to act trivially on $\mathbb Z^{4(n + 1)}$ and so necessarily $\gamma = \id$.\\
	Since $L^\infty(\mathbb T^{4(n + 1)}) = \Aa_n \subset \emm_n = L^\infty(\mathbb T^{4(n + 1)})  \rtimes \Lambda_n$ is a Cartan subalgebra and since $\emm_n$ is a type ${\rm II_1}$ factor, the probability measure-preserving action $\Lambda_n \curvearrowright \mathbb T^{4(n + 1)}$ is essentially free and ergodic. Then \cite[Theorem 1.1]{PV12} shows that $\Aa_n \subset \emm_n$ is the unique Cartan subalgebra, up to unitary conjugacy. 
\end{proof}
\subsection{Class $\mathscr S$}\label{bo}

Using the powerful Dehn filling technology from \cite{Os06}, Belegradek and Osin showed in \cite[Theorem 1.1]{BO06} that for every finitely generated group $Q$ one can find a property (T) group $N$ such that $Q$ embeds into ${\rm Out}(N)$ as a finite index subgroup. This canonically gives rise to an action $Q \curvearrowright^\rho N$ by automorphisms such that the corresponding semidirect product group $N\rtimes_\rho Q$ is hyperbolic relative to $\{Q\}$. Throughout this document the semidirect products $N\rtimes_\rho Q$ will be termed Belegradek-Osin's Rips construction groups. When $Q$ is torsion free then one can pick $N$ to be torsion free as well and hence both $N$ and $N\rtimes_\rho Q$ are icc groups. Also when $Q$ has property (T) then $N\rtimes_\rho Q$ has property (T). Under all these assumptions we will denote by $\mathcal Rips(Q)$ the class of these Rips construction groups $N\rtimes_\rho Q$. 

In \cite[Sections 3,5]{CDK19} we introduced a class of property (T) groups based on the Belegradek-Osin Rips construction groups and we have proved several rigidity results for the corresponding von Neumann algebras, \cite[Theorem A]{CDK19}. Next we briefly recall this construction also highlighting its main algebraic properties that are relevant in the proofs of our main results in the next section.

\vskip 0.05in 
\noindent  {\bf Class $\mathscr S$}.
\label{semidirectt} Consider any product group $Q= Q_1\times Q_2$, where $Q_i$ are any nontrivial, bi-exact, weakly amenable, property (T), residually finite, torsion free, icc groups. Then for every $i=1,2$  consider a Rips construction  $G_i = N_i \rtimes_{\rho_i} Q\in \mathcal Rips(Q)$, let $N=N_1\times N_2$  and denote by $G= N\rtimes_\sigma Q$ the canonical semidirect product which arises from the diagonal action  $\sigma=\rho_1\times \rho_2: Q\rar {\rm Aut}(N)$, i.e. $\sigma_g (n_1,n_2)=( (\rho_1)_g(n_1), (\rho_2)_g(n_2))$ for all $(n_1,n_2)\in N$. Throughout this article the category of all these  semidirect products $G$ will denoted by {\bf Class $\mathscr S$}.

Concrete examples of semidirect product groups in class  $\mathscr S$ can be obtained if the initial groups $Q_i$ are any uniform lattices in $Sp(n,1)$ when $n\geq 2$. Indeed one can see that required conditions on $Q_i$'s follow from \cite{Oz03,CH89}.

For further reference we record some algebraic properties of groups in class $\mathscr S$. For their proofs the reader may consult \cite[Sections 3,4,5]{CDK19} and the references within. 

\begin{theorem}\label{algprop}
	For any $G= N\rtimes_\sigma Q \in \mathscr S$ the following hold
	\begin{itemize}
		
		\item[a)] $G$ is an icc, torsion free, property (T) group;
		
		\item[b)] $Q$ is malnormal subgroup of $G$, i.e.\ $gQg^{-1}\cap Q =\{e\}$ for every $g\in G\setminus Q$;
		
		\item[c)] The stabilizer   ${\rm Stab}_Q(n)=\{e\}$ for every $n \in N\setminus \{e\}$;
		\item[d)] The virtual centralizer satisfies  $vC_G(N)=1$;
		\item [e)] $G$ is the fiber product $G=G_1\times_Q G_2$; thus embeds into $G_1\times G_2$ where $Q$ embeds diagonally into $Q\times Q$.

	\end{itemize}
\end{theorem}

Finally we conclude this section with a folklore lemma related to the calculation of centralizers of elements in products of hyperbolic groups. We include some details for readers' convenience.
\begin{lemma} \label{centralizerstructure}
	Let $Q=Q_1\times Q_2$, where $Q_i$s are non-elementary torsion free, hyperbolic groups. For any $e\neq g\in Q$ the centralizer $C_Q(g)$ is of one of the following forms: $A$, $A \times Q_2 $ or $Q_1 \times A $, where $A$ is an amenable group. 
\end{lemma} 
\begin{proof}
	Let $g=(g_1,g_2)\in Q$ where $g_i\in Q_i$ and notice that $C_Q(g)=C_{Q_1}(Q_1)\times C_{Q_2}(g_2)$. Therefore to get our conclusion it suffices to show that for every $g_i\in Q_i$ either $C_{Q_i}(g_i)=Q_i$ or $C_{Q_i}(g_i)$ is an elementary group. However this is immediate once we note that for every $g_i\neq e$ the centralizer satisfies  $C_{Q_i}(g_i)\leqslant E_{Q_i}(g_i)$, where $E_{Q_i}(g_i)$ is maximal elementary subgroup containing $g_i$ of the torsion free icc hyperbolic group $Q_i$, see for example \cite{Ol91}.
\end{proof}


\section{Fundamental Group of Factors Arising from Class $\mathscr S$}
In this section we prove our main result describing  isomorphisms of amplifications of property (T) group factors $\El(G)$ associated with groups $G \in \mathscr S$. These factors were first considered  in \cite{CDK19}, where various rigidity properties were established. For instance, in  \cite[Theorem A]{CDK19} it was shown that the semidirect product decomposition of the group $G=N\rtimes Q$ is a feature that's completely recoverable from $\El(G)$. In this section we continue these investigations by showing in particular that these factors also have trivial fundamental group (see Theorem \ref{mainthm} and Corollary ~\ref{fg1}). In order to prepare for the proof of our main theorem we first need to establish several preliminary results on classifying specific subalgebras of $\mathcal L(G)$. Some of the theorems will rely on results proved in \cite{CDK19}. We recommend the reader to consult these results beforehand as we will focus mostly on the new aspects of the techniques. Throughout this section we shall use the notations introduced in Section~\ref{bo}.

\vskip 0.03in
Our first result classifies all diffuse, commuting property (T) subfactors inside these group factors.
\begin{theorem}\label{thm2}
	Let $N\rtimes Q\in \mathscr S$. Also let $\Aa_1,\Aa_2\subseteq \El(N\rtimes Q)=\emm$ be two commuting, property (T), type $\rm II_1$ factors. Then for all $k\in\{1,2\}$ one of the following holds:
	\begin{enumerate}
		\item There exists $i\in \{1,2\}$ such that $\Aa_i\prec_{\emm} \El(N_k)$;
		\item $\Aa_1\vee \Aa_2\prec_{\emm} \El(N_k)\rtimes Q$. 
	\end{enumerate}
\end{theorem}
\begin{proof}
	Let $G_k= N_k \rtimes Q$ for $k \in \{1,2\}$. Notice that by part e) in Theorem \ref{algprop} we have that $N\rtimes Q\leqslant G_1\times G_2=G$ where $Q$ is embedded as  ${\rm diag}(Q)\leqslant Q\times Q$. Notice that $\Aa_1,\Aa_2\subseteq \El(N)\rtimes Q\subseteq \El(G_1\times G_2)=:\tilde{\emm}$. By \cite[Theorem~5.3]{CDK19} there exists $i\in \{1,2\}$ such that 
	\begin{enumerate}
		\item[a)] $\Aa_i\prec_{\tilde{\emm}} \El(G_k)$, or
		\item[b)] $\Aa_1\vee\Aa_2\prec_{\tilde{\emm}}\El(G_k\times Q)$.
	\end{enumerate}
	Assume a). Since $\Aa_1\vee\Aa_2\subseteq \El(N)\rtimes Q$, by using \cite[Lemma~2.3]{CDK19} we further get that $\Aa_i\prec_{\mtil}\El(G\cap h G_kh^{-1})=\El(((N_1\times N_2)\rtimes {\rm diag}(Q))\cap (N_k\rtimes Q))=\El(N_k)$ and thus we have that $c)\ \Aa_i\prec_{\mtil}\El(N_k)$. 
	
	Assume b). Then $\Aa_1\vee\Aa_2\prec_{\mtil}\El(\G\cap h(\G_k\times Q)h^{-1})=\El(h(N_k\rtimes {\rm diag}(Q))h^{-1})$. This implies that $d) \ \Aa_1\vee\Aa_2\prec_{\mtil}\El(N_k)\rtimes Q$.
	
	Note that by using \cite[Lemma~2.5]{CDK19} case d) already implies that $\Aa_1\vee\Aa_2\prec_{\emm}\El(N_k)\rtimes Q$ which gives possibility 2. in the statement.
	
	Next we show that c) gives 1. To accomplish this we only need to show that the intertwining actually happens in $\emm$. By Popa's intertwining techniques c) implies there exist finitely many $x_i\in \tilde\emm$, and $c>0$ such that 
	\begin{align}\label{eqn:1}
	\sum_{i=1}^{n} \|E_{\El(N_k)}(ax_i)\|_2^2 \geq c\text{  for  all }  a\in\mathscr U(\Aa_i). 
	\end{align}    
	
	Using basic approximations of $x_i$'s and increasing $n\in \mathbb N$ and decreasing $c>0$, if necessary, we can assume that $x_i=u_{g_i}$ where $g_i\in \hat{G}_k\times Q$. Now observe that $E_{\El(N_k)}(ax_i)=E_{\El(N_k)}(au_{g_i})=E_{\El(N_k)}(E_{\emm}(au_{g_i}))=E_{\El(N_k)}(aE_{\emm}(u_{g_i}))$. Thus (\ref{eqn:1}) becomes 
	\begin{align}
	\sum_{i=1}^{n}\|E_{\El(N_k)}(aE_{\emm}(u_{g_i}))\|^2_2\geq c \ for\ all\ u\in\mathscr U(\Aa_i)\nonumber
	\end{align}  
	and hence $\Aa_i\prec_{\emm}\El(N_k)$ as desired.\end{proof}
Next we show that actually the intertwining statements in the previous theorem can be made much more precise.  
\begin{theorem} \label{thm3}	Let $N\rtimes Q\in \mathscr S$. Also let $\Aa_1,\Aa_2\subseteq \El(N\rtimes Q)=\emm$ be two commuting, property (T), type $\rm II_1$ factors. Then for every $k \in \{1,2\}$ one of the following holds:
	\begin{enumerate}
		\item There exists $i \in \{1,2\}$ such that $\Aa_i \prec_{\emm} \El(N_k)$;
		
		\item $\Aa_1 \vee \Aa_2 \prec_{\emm} \El(Q)$.
	\end{enumerate}
\end{theorem}

\begin{proof}
	Using Theorem~\ref{thm2} the statement will follow once we show that $\Aa_1 \vee \Aa_2 \prec_{\emm} \El(N_k) \rtimes Q$ implies $\Aa_1 \vee \Aa_2 \prec_{\emm} \El(Q)$, which we do next. Since $\Aa_1 \vee \Aa_2 \prec_{\emm} \El(N_k)\rtimes Q $, there exists \begin{equation} \label{eq1}
	\psi: p(\Aa_1 \vee \Aa_2)p \rightarrow \psi( p(\Aa_1 \vee \Aa_2)p ) = \mr \subseteq q (\El(N_k) \rtimes Q)q
	\end{equation}
	$\ast$-homomorphism, nonzero partial isometry $v \in q \emm p$ such that \begin{align}\label{bal2}\psi(x)v =vx \text{ for all }x \in p(\Aa_1 \vee \Aa_2)p.\end{align}
	Notice that we can pick $v$ such that the support projection satisfies  $s(E_{\El(N_k\rtimes Q)}(vv^*))=q$. Moreover, since $\Aa_i$'s are factors we can assume that $p=p_1p_2$ for some $p_i\in \mathscr P(\Aa_i)$. 
	
	Next let $\mr_i= \psi(p_i \Aa_i p_i)$. Note that $\mr_1$, $\mr_2$ are commuting property (T) subfactors such that $\mr_1 \vee \mr_2= \mr \subseteq q (\El(N_k) \rtimes Q)q$. Using the Dehn filling technology from $\cite{Os06,DGO11}$, we see that there exists a short exact sequence $1\rar \underset{\g_j}{\ast} Q_0^{\g_j} \rar N_k\rtimes Q\rar H \rar 1$ where $H$ is a hyperbolic, property (T) group and $Q_0\leqslant Q$ is a finite index subgroup.  Then using $\cite{PV12, CIK13}$ in the same way as in the proof of \cite[Theorem 5.2]{CDK19} we have either a) $\mr_i \prec_{\El(N_k)\rtimes Q} \El(\underset{\g_j}{\ast} Q_0^{\g_j})$, for some $i$, or b) $\mr=\mr_1\vee \mr_2 \lessdot_{\El(N_k)\rtimes Q} \El(\underset{\g_j}{\ast} Q_0^{\g_j})$. Since $\mr_i$'s have property (T) then by \cite[Proposition 4.6]{Po01} so does $\mathcal R$ and hence possibility b) entails $\mr \prec_{\El(N_k)\rtimes Q} \El(\underset{\g_j}{\ast} Q_0^{\g_j})$. Summarizing, cases a)-b) imply that $\mr_i \prec_{\El(N_k)\rtimes Q} \El(\underset{\g_j}{\ast} Q_0^{\g_j})$, for some $i$. Then using \cite[Theorem 4.3]{IPP05} this further implies $\mr \prec_{\El(N_k)\rtimes Q} \El(Q_0^{\g_j})$ and hence $\mr_i \prec_{\El(N_k)\rtimes Q} \El(Q_0) \subseteq \El(Q)$. As $Q \leqslant N_k \rtimes Q$ is malnormal, using the same arguments as in the proof of \cite[Theorem 5.3]{CDK19} one can show that $\mr  \prec_{\El(N_k) \rtimes Q} \El(Q)$. Indeed, let 
	$\phi: r\mr_i r \rightarrow  \phi( r\mr_i r ):=\mrt \subseteq q_1 \El(Q) q_1$
	be a unital $\ast$-homomorphism, and let $w \in q_1 \El(N_k\rtimes Q) r$ be a nonzero partial isometry such that 	\begin{equation}\label{bal}
	\phi(x)w=wx \text{ for all } x \in r \mr_i r.
	\end{equation} Note that $ww^{\ast} \in  \El(Q) $ by Lemma ~\ref{malnormalcontrol} and hence $\mrt ww^{\ast}=w\mr_i w^{\ast} \subseteq \El(Q)$. For every $u \in \mr_{i+1}$ we have \begin{align*}
	\mrt wuw^{\ast} &=   \mrt ww^{\ast} wuw^{\ast} = w \mr_i w^{\ast} wuw^{\ast}= ww^{\ast} wu \mr_i w^{\ast} = w u\mr w^{\ast} \\
	&= w u\mr_i w^{\ast} w w^{\ast}= wuw^{\ast} w \mr_i w^{\ast}= wuw^{\ast} \mrt ww^{\ast}= wuw^{\ast}\mrt.
	\end{align*}
	
	Thus Lemma \ref{malnormalcontrol} again implies that $wuw^{\ast} \in \El(Q)$. Altogether these show that $w \mr_{i+1}w^{\ast} \subseteq \El(Q)$. Combining with the above we get $w\mr w^*=w \mr_i \mr_{i+1} w^{\ast}=w w^{\ast}w \mr_i\mr_{i+1}w^{\ast}= w \mr_i w^{\ast}w \mr_{i+1}w^{\ast} \subseteq \El(Q)$. 
	From relation \eqref{bal} we have that $w^*w\in \mr$. Also by \eqref{bal2} we have $\mr v = v p(\Aa_1\vee \Aa_2) p$ and hence $v^*\mr v = v^*v p(\Aa_1\vee \Aa_2) p$. Hence there exists $p_0\in\mathscr P(p(\Aa_1\vee \Aa_2)p)$ so that $v^*w^*wv=v^*v p_0$. Next we argue that $wvp_0\neq 0$. Indeed, otherwise we would have $wv=0$ and hence  $wvv^* = 0$. As $w\in \El(N_k\rtimes Q)$ this would imply that $wE_{\El(N_k\rtimes Q)}(vv^*) = 0$ and hence $w= wq =w s(E_{\El(N_k\rtimes Q)}(vv^*)) = 0$, which is a contradiction. To this end, combining the previous relations we have $wv p(\Aa_1\vee \Aa_2)p p_0\subseteq wv p(\Aa_1\vee \Aa_2)p v^*vp_0=wv p(\Aa_1\vee \Aa_2)p v^*w^*wv= w\mr vv^* w^*wv= w\mr w^* wv\subseteq \El(Q)wv$. Since the partial isometry $ wv\neq 0 $ the last relation clearly shows that $\Aa_1\vee \Aa_2\prec_\emm \El(Q)$, as desired. \end{proof} 

\begin{theorem} \label{intertwiningthm}
	Let $\Aa_1,\Aa_2\subseteq \El(N)\rtimes Q=\emm$ be two commuting, property (T), type $\rm II_1$ factors such that $(\Aa_1\vee\Aa_2)'\cap r(\El(N)\rtimes Q)r=\mathbb{C}r$. Then one of the following holds:
	\begin{enumerate}
		\item[a)] $\Aa_1\vee\Aa_2\prec^s_{\emm}\El(N)$, or
		\item[b)] $\Aa_1\vee\Aa_2\prec^s_{\emm}\El(Q)$.
	\end{enumerate}
\end{theorem}
\begin{proof}
	Fix $k\in\{1,2\}$. By Theorem~\ref{thm3} we get that either 
	\begin{enumerate}
		\item[i)] $i_k\in\{1,2\}$ such that $\Aa_{i_k}\prec_{\emm} \El(N_k)$, or
		\item[ii)] $\Aa_1\vee\Aa_2\prec_{\emm} \El(Q)$.
	\end{enumerate}
	Note that case ii) together with the assumption $(\Aa_1\vee\Aa_2)'\cap r(\El(N)\rtimes Q)r=\mathbb{C}r$ and \cite[Lemma 2.4]{DHI16} already give $b)$. So assume that case i) holds. Hence for all $k\in\{1,2\}$, there exists $i_k\in\{1,2\}$ such that $\Aa_{i_k}\prec_{\emm}\El(N_k)$. Using \cite[Lemma~2.4]{DHI16}, there exists $0\neq z\in\mathcal Z(\enn_{r\emm r}(\Aa_{i_k})'\cap r\emm r)$ such that $\Aa_{i_k}z \prec^s_{\emm}\El(N_k)$. Since $\Aa_1\vee\Aa_2\subseteq \enn_{r\emm r}(\Aa_{i_k})''$, then $\enn_{r\emm r}(\Aa_1{i_k})'\cap r\emm r\subseteq (\Aa_1\vee\Aa_2)'\cap r\emm r=\mathbb{C}r$. Thus we get that $z=r$. In particular 
	\begin{align}\label{1}
	\Aa_{i_k}\prec^s_{\emm} \El(N_k).  
	\end{align}
	We now briefly argue that $k\neq l\Rightarrow i_k\neq i_l$. Assume by contradiction that $i_1=i_2=i$. Then (\ref{1}) implies that $\Aa_i\prec^s_{\emm}\El(N_1)$ and $\Aa_i\prec^s_{\emm}\El(N_2)$. By \cite[Lemma 2.6]{DHI16}, this implies that $\Aa_i\lessdot_{\emm} \El(N_1)$ and $\Aa_i\lessdot_{\emm} \El(N_2)$. Note that $\El(N_i)$ are regular in $\emm$ and hence by \cite[Proposition 2.7]{PV11} we get that $\Aa_i\lessdot_{\emm} \El(N_1)\cap\El(N_2)=\mathbb{C}$, which implies that $\Aa_i$ is amenable. This contradicts our assumption that $\Aa_i$ has property $(T)$. Thus $i_k\neq i_l$ whenever $k\neq l$. Therefore we have that $\Aa_{i_1}\prec_{\emm}^s\El(N_1)\subseteq \El(N)$ and $\Aa_{i_2}\prec_{\emm}^s\El(N_2)\subseteq \El(N)$. Using Corollary~\ref{cor:join} we get that $\Aa_1\vee\Aa_2\prec^s_{\emm}\El(N)$, which completes the proof.  
\end{proof}

Our next result concerns the location of the "core" von Neumann algebra.

\begin{theorem} \label{coreunit} Let $N\rtimes Q, M\rtimes P \in \mathscr S$. Let $p\in \El(M\rtimes P)$ be a projection and assume that  $\Theta: \El(N \rtimes Q) \rightarrow p\El(M \rtimes P)p$ is a $\ast$-isomorphism.  Then there exists a unitary $v \in \mathscr U(p\El(M \rtimes P)p)$ such that $\Theta(\El(N))= vp\El(M)pv^{\ast}$. 
\end{theorem}

\begin{proof} From assumptions there are $Q_1$, $Q_2$, $P_1$, $P_2$ icc, torsion free, residually finite, hyperbolic property (T) groups so that $Q=Q_1 \times Q_2$ and $P=P_1 \times P_2$. We also have that $N=N_1\times N_2$ and $M=M_1\times M_2$ where $N_i$'s and $M_i$'s have property (T). Denoting by $\emm= \mathcal L(M\rtimes P)$, $\Aa =  \Theta(\mathcal L(N))$ and $\Aa_i = \Theta(\mathcal L(N_i))$ we see that  $\Aa_1$ and $\Aa_2$ are commuting property (T) subalgebras of $p \emm p$. Using part b) in Theorem \ref{intertwiningthm} we have that  $\{\Aa_1\vee \Aa_2\}'\cap \enn = \Theta (\mathcal L(N)'\cap \El(N\rtimes Q))= \mathbb C \Theta(1)= \mathbb Cp$. Using Theorem ~\ref{intertwiningthm} we get either
	\begin{enumerate}
		\item [a)]$\Aa \prec^s_{\emm} \El(M)$ or,
		
		\item [b)] $\Aa \prec^s_{\emm} \El(P)$.
	\end{enumerate}
	Assume case b) above holds. 
	Then there exists projections $r \in \Aa$, $q \in \El(P)$, a nonzero partial isometry $v \in q \emm r$, and a $\ast$-homomorphism $\psi:r\Aa r \rightarrow \psi( r \Aa r) \subseteq q \El(P)q$ such that $\psi(x)v=vx$ for all $x \in r \Aa r$. 
	Arguing exactly as in the proof of \cite[Theorem 5.5]{CDK19}, we can show that $v \mathcal{QN}_{r \emm r}(r \Aa r)''v^{\ast} \subseteq  q\El(P) q$. 
	
	Now, $\mathcal{QN}_{r \emm r}(r \Aa r)''= r \emm r$, using \cite[Lemma 3.5]{Po03}. Thus, $\emm \prec_{\emm} \El(P)$ and hence $\El(P)$ has finite index in $\emm$ by \cite[Theorem 2.3]{CD18}, which is a contradiction. Hence we must a), i.e.\ $\Aa \prec^s_{\emm} p \El(M) p$.		
	
	Repeating the above argument verbatim, we get that $p \El(M) p \prec_{p \emm p} \Aa$.	
	Let $\enn =\El(N\rtimes Q)$ and $\mathcal B= p\El(M)p$. Note that $\Aa \subseteq \Theta(\enn) $ and $\bee \subseteq p \emm p$ are amplifications of genuine crossed product inclusions. Also by part d) in Theorem \ref{algprop}  $\Aa$ is regular irreducible subfactor of $\Theta (\enn)=p \emm p$, while $\bee$ is a quasi-regular irreducible subfactor of $p \emm p$ (as $\mathcal{QN}_{p \emm p}(p\bee p)''= p\mathcal{QN}_{\emm}(\El(M))p $). Thus, we are in the setting of the first part of the proof of \cite[Lemma 8.4]{IPP05} and using the same arguments there we conclude one can find $r \in \mathscr P(\Aa)$, a unital $\ast$-isomorphism $\psi: r \Aa r \rightarrow \mathcal R:=\psi( r \Aa r) \subseteq p \El(M) p$, and a partial isometry $v \in p \emm p$	satisfying $v^{\ast}v =r$, $vv^{\ast} \in R' \cap p \emm p$ and  $\psi(x)v=vx$ for all $x\in r\Aa r$. Moreover, we have that $\mathcal R \subseteq p \El(M) p$ has finite index, and $\mathcal R ' \cap p \El(M) p = \mc p$.  
	Notice that by \cite[Lemma 3.1]{Po02}, we have that $[R' \cap p \emm p: (p \El(M)p)' \cap p \emm p] \leq [p \El(M) p : \mathcal R]$. As $(p \El(M)p)' \cap p \emm p= \mc$, we conclude that $\mathcal R' \cap p \emm p$ is finite dimensional.
	
	Let $x\in \mr'\cap p\emm p$. Since  $xr=rx$ for all $r\in \mr$ we have that $r{\sum}_gx_gu_g={\sum}_g x_g u_g r$, where $x={\sum}_{g\in P}x_gu_g$ is the Fourier decomposition of $x$ in $\emm= \El(M)\rtimes P$. Thus ${\sum}_g rx_gu_g={\sum}_gx_g\sigma_g (r)u_g$ and hence $ rx_g=x_g\sigma_p(r)$ for all $g$ in $r$ . In particular this entails that  
	\begin{align}\label{commutant} 
	&x_gx_g^*\in \mr'\cap p\El(M) p=\mc p
	\\
	&\label{eq:2}
	x_gu_g\in \mr'\cap p\emm p.
	\end{align}  
	From (\ref{commutant}) we see that $x_g$ is a scalar multiple of a unitary in $p\emm p$. Hence by normalization we may assume that each $x_g$ is itself either a unitary or zero. 
	
	Let $K$ be the set of all $g\in P$ for which there exists   $x_g \in \mathscr U (p\El(M)p)$ such that  $x_{g}u_{g}\in \mathscr U (\mr'\cap p\emm p)$ and notice that $K$ is a subgroup of $P$. Note that $\{x_{g}u_{g} \}_{g \in K}$ is a $\tau$-orthogonal family in $\mr'\cap p\emm p$. As $\mr'\cap p \emm p$ is finite dimensional, we get that $K$ is a finite subgroup of $P$.   As $P$ is torsion free (see part a) in Theorem \ref{algprop}) then $K=\{e\}$. In particular this shows that $\mathcal R'\cap p\emm p= \mathcal R'\cap p\El(M)p = \mathbb C p$ which implies $vv^*=p$ and since $v^*v=r \leq p$ we get $r=p$ and  $v\in \mathscr U(p\emm p)$. Thus $\psi(x)=vxv^*$ for all $x\in r \Aa r$ and  hence $\mathcal R=v r \Aa r v^{\ast}= v\Aa v^* \subseteq p \El(M) p$.  Let $v= \Theta(w_0)$, where $w_0 \in \euu(\El(N \rtimes Q))$. Thus, we get that $\El(N) \subseteq w_0^{\ast} \Theta^{-1}(p \El(M)p) w_0 \subseteq \El(N) \rtimes Q$. By, \cite{Ch78} (see also \cite[Corollary 3.8]{CD19}), we deduce that there exists a subgroup $L \leqslant Q$ such that $w_0^{\ast} \Theta^{-1}(p \El(M)p) w_0= \El(N) \rtimes L$. As $[w_0^{\ast} \Theta^{-1}(p \El(M)p) w_0: \El(N)]$ is finite, we must have that $L$ is a finite subgroup of the torsion free group $Q$. Thus $L=\{e\}$ which gives that $\Theta(\El(N))=\Aa= v^{\ast} p\El(M)p v$.\end{proof}


Next we show that in the previous result we can also identify up to corners the algebras associated with the acting groups. The proof relies heavily on the classification of commuting property (T) subalgebras provided by  \ref{intertwiningthm} and the malnormality of the acting groups. 
\begin{theorem} \label{actinggp} 
	Let $N\rtimes Q, M\rtimes P \in \mathscr S$. Let $p\in \El(M\rtimes P)$ be a projection and assume that  $\Theta: \El(N \rtimes Q) \rightarrow p\El(M \rtimes P)p$ is a $\ast$-isomorphism.  Then the following hold
	
	\begin{enumerate}
		\item There exists  $v \in \mathscr U (p\El(M \rtimes P)p)$ such that $\Theta(\El(N))= vp\El(M)pv^{\ast}$, and 
		\item There exists $u \in \mathscr U (\El(M \rtimes P))$ such that $\Theta(\El(Q))=pu^{\ast} \El(P)up$.
	\end{enumerate}
	
\end{theorem}

\begin{proof} As part 1. follows directly from Theorem \ref{coreunit} we only need to show part 2. 
	
	Recall that $Q=Q_1 \times Q_2$, $P=P_1 \times P_2$, $N=N_1\times N_2$ and $M=M_1\times M_2$ where $Q_i$, $P_i$, $N_i$ and $M_i$ are icc, property (T) groups. Denote by $\emm= \mathcal L(M\rtimes P)$, $\Aa=\Theta(\El(N))$, $\bee =  \Theta(\mathcal L(Q))$ and $\bee_i = \Theta(\mathcal L(Q_i))$. Then we see that  $\bee_1, \bee_2\subset p\emm p$ are commuting property (T) subalgebras such that $\bee_1\vee \bee_2=\bee$. Moreover, by part d) in Theorem \ref{algprop} we have that $\{\bee_1\vee \bee_2\}' \cap p\emm p=\bee' \cap \Theta(\El(N \rtimes Q)) =\mathbb C \theta (1)=\mathbb C p$.  Hence by Theorem~\ref{intertwiningthm}, we either have that a) $\bee \prec^s_{\emm} \El(M)$, or b) $\bee \prec^s_{\emm} \El(P)$. By part 1. we also know that $\Aa \prec_{\emm}^s \El(M)$. Thus, if a) holds, then Theorem \ref{joinintertwining} implies that $p\emm p=\Theta(\El(N\rtimes Q)) \prec_{\emm} \El(M)$. In turn this implies that $Q$ is finite, a contradiction. Hence b) must hold, i.e.\ $\bee \prec^s_{\emm} \El(P)$.
	\vskip 0.05in 
	Thus there exist projections $q \in \bee$, $r \in \El(P)$, a nonzero partial isometry $v \in \emm$ and a $\ast$-homomorphism $\psi: q\bee q \rightarrow \mathcal R:=\psi(q\bee q) \subseteq r \El(P)r$ such that $\psi(x)v=vx$ for all $x \in q\bee q $. Note that $vv^{\ast} \in R' \cap r \emm r$. Since $\mathcal R \subseteq r \El(P)r$ is diffuse, and $P \leqslant  M \rtimes P$ is a malnormal subgroup (part c) in Theorem \ref{algprop}), we have that $\mathscr{QN}_{r \emm r}(\mathcal R)'' \subseteq r \El(P)r$. Thus $vv^{\ast} \in r \El(P)r$ and hence $vq\bee qv^{\ast}= \mathcal Rvv^{\ast} \subseteq r \El(P)r$. Extending $v$ to a unitary $v_0$ in $\emm$ we have that $v_0q\bee qv_0^{\ast} \subseteq \El(P)$. As $\El(P)$ and $\bee $ are factors, after perturbing $v_0$ to a new unitary $u$, we may assume that $u \bee u^{\ast} \subseteq \El(P)$. This further implies that $upu^*\in \El(P)$ and  since $\Theta(1)=p$ we also have \begin{align} \label{eqn1}
	\bee =p\bee p\subseteq pu^*\El(P)up.
	\end{align} Next we claim that 
	
	\begin{align}\label{cont1}pu^*\El(P)up\prec_{\emm} \bee\end{align}
	To see this first notice that,  since $P$ is malnormal in $M\rtimes P$ and $P$ is icc  (see parts a) and c) in Theorem \ref{algprop}) then $(pu^*\El(P)up)'\cap \Theta(\El(N\rtimes Q))=(pu^*\El(P)up)'\cap p\emm p= u^*(\El(P)'\cap \El(M\rtimes Q))u p=\mathbb C p$. Thus using Theorem \ref{intertwiningthm} we have either a) $pu^*\El(P)up\prec^s_{p\emm p} \Aa$ or b) $pu^*\El(P)up \prec^s_{p\emm p} \bee$. Assume a) holds. By part 1. we have $pu^*\El(P)up\prec^s_{p\emm p} \Aa= v p \El(M)pv^*$; in particular, this implies that $\El(P)\prec_\emm \El(M)$ but this contradicts the fact that  $\El(M)$ and $\El(P)$ are diffuse algebras that are $\tau$-perpendicular in $\emm$. Thus b) holds which proves the claim.
	
	Using \eqref{cont1} together with malnormality of $\Theta (\El(Q))$ inside $\Theta(\El(N\rtimes Q))$ and arguing exactly as in the proof of relation \eqref{eqn1} we conclude that there exists $w\in \mathcal{U}(p\emm p)$ such that 
	\begin{align} \label{eqn2}
	wpu^*\El(P)upw^* \subseteq \bee.
	\end{align} 
	Combining (~\ref{eqn1}) and (~\ref{eqn2})  we get that $w\bee w^*\subseteq wpu^*\El(P)upw^*\subseteq \bee$ and hence $w\in \mathscr{Q}\mathcal{N}_{p\emm p}(\bee)"=\bee$. Thus we get 
	\begin{align}\label{eqn3}
	pu^*\El(P)up\subseteq w^*\bee w=\bee.
	\end{align}
	Combining (\ref{eqn1}) and $(\ref{eqn3})$ we get the theorem. 
\end{proof}
Finally, we are now ready to derive the main result of this paper.
\begin{theorem} \label{mainthm}
	Let $N\rtimes Q, M\rtimes P \in \mathscr S$ with $N=N_1\times N_2$ and $M=M_1\times M_2$. Let $p\in \El(M\rtimes P)$ be a projection and assume that  $\Theta: \El(N \rtimes Q) \rightarrow p\El(M \rtimes P)p$ is a $\ast$-isomorphism. Then $p=1$ and one can find  $\ast$-isomorphisms, $\Theta_i: \El(N_i) \rightarrow \El(M_i)$, a group isomorphism $\delta: Q \rightarrow P$, a multiplicative character $\eta: Q \rightarrow \mathbb T$, and a unitary $u \in \mathscr U( \El(M \rtimes P))$ such that  for all $ g \in Q$, $x_i \in N_i$ we have that$$ \Theta((x_1 \otimes x_2)u_{g})= \eta(g)u(\Theta_1(x_1) \otimes \Theta_2(x_2)v_{\delta(g)})u^{\ast}. $$
\end{theorem}
\begin{proof} Throughout this proof we will denote by $\emm=\El(N\rtimes Q)$. Using Theorem~\ref{coreunit}, and replacing $\Theta$ by $\Theta \circ Ad(v)$ if necessary, we may assume that $\Theta(\El(N))= p\El(M)p$. By Theorem~\ref{actinggp}, there exists $u \in \mathscr U(\emm)$ such that $ \Theta(\El(Q)) \subseteq u^{\ast}\El(P)u$, where $\emm= \El(M \rtimes P)$. Moreover $\Theta(1)=p$, $upu^{\ast} \in \El(P)$ and also $\Theta(\El(Q))= pu^{\ast}\El(P)up$. Next we denote by $\G= u^{\ast} P u$ and by $\mathcal G =\{\Theta (u_g)\, :\,g\in Q\}$. Using these notations we show the following  
	
	\begin{claim} \label{ht}
		$h_{\G}(\mathcal G)>0.$
	\end{claim}
	
	\noindent \textbf{Proof of Claim ~\ref{ht}}. Notice that $\mathcal G \subseteq \El(\G)$ is a group of unitaries normalizing $\Theta(\El(N))$. Moreover, by Theorem \ref{algprop} we can see that the action $\sigma: P\rar {\rm Aut}( M )$ satisfies all the conditions in the hypothesis of Theorem~\ref{height} and thus using the conclusion of the same theorem we get the claim. $\hfill \blacksquare$
	\begin{claim}\label{noint}
		Let $e\neq g \in \G$. Then $\mathcal G'' \nprec \El(C_{\G}(g))$.
	\end{claim}
	
	\noindent \textbf{Proof of Claim \ref{noint}}. Since $\G$ is isomorphic to the product of two biexact groups, say $\G_1 \times \G_2$, by Lemma~\ref{centralizerstructure} we get that $C_{\G}(g)= A$, $\G_1 \times A$, or $A\times \G_2$ for an amenable group $A$. If $C_{\G}(g)= A$ then since  $\mathcal G$ is non-amenable we clearly have $\mathcal G'' \nprec \El(C_\G(g))$.  Next assume $C_{\G}(g)= A\times \G_2$ and assume by contradiction that $\mathcal G''\prec \El(C_\G(g)$. As $Q=Q_1\times Q_2$ for $Q_i$ property (T) icc group, then $\mathcal G'' =\Theta (\El(Q_1)) \bar\otimes  \Theta(\El( Q_2))$ is a $\rm II_1$ factor with property (T). Since $\mathcal G '' \prec \El(A\times \G_2)=\El(A)\bar\otimes \El(\G)$  and $\El(A)$ is amenable then it follows that $\mathcal G '' \prec \El(\G)$. However by \cite[Theorem 1]{Oz03} this is impossible as $\El(\G_2)$ is solid and $\mathcal G''$ is generated by two non-amenable commuting subfactors. The case  $C_\G(g)=\G_1\times A$ follows similarly. $\hfill\blacksquare$
	
	\begin{claim}\label{wmix}
		The unitary representation $\{{\rm Ad}(v)\}_{v \in \mathcal G}$ on $L^2(p \El(\G)p \ominus \mc p)$ is weakly mixing.
	\end{claim}
	
	\noindent \textbf{Proof of Claim \ref{wmix}}. First note we have that $ \Theta(\El(Q))=\mathcal G''=p\El(\G)p$.  Also since $Q$ is icc then using \cite[Proposition 3.4]{CSU13} the representation  ${\rm Ad} (Q)$ on $L^2(\El(Q)\ominus \mc)$ is weak mixing. Combining these two facts, we get that the representation $\mathcal G$  on $L^2(p \El(\G)p \ominus \mc p)$ is weak mixing, as desired.$\hfill\blacksquare$
	\vskip 0.05in
	Claims 2-4 above together with Theorem \ref{kv15} show that $p=1$ and moreover there exists  unitary $w\in \El(M\rtimes P)$, a group isomorphism $\delta: Q\rar P$ and a multiplicative character $\eta: Q\rar \mathbb T$ such that $\Theta (u_g) = \eta(g) w v_{\delta(g)}w^*$ for all $g\in Q$. Since $\Theta(\El(N))=\El(M)$ then the same argument as in proof of \cite[Theorem 4.5]{CD19} (lines 10-27 on page 25) shows that i) $w^*\El(M)w\subseteq \El(M)$. However re-writing the previous relation as  $v_{h} =\overline{\eta(g)} w^* \Theta (u_{\delta^{-1}(h)}) w $ for all $h\in P$ and applying the same argument as above for the decomposition $\emm=\Theta(\El(N)) \rtimes \Theta (Q)$ we get that ii) $w\Theta (\El(N))w^*\subseteq \Theta(\El(N))$. Then combining i) and ii) we get that $w^*\El(M)w= \El(M)$. Now from the above relations it follows clearly that the map $\Psi={\rm ad}(w^*)\circ \Theta: \El(N)\rar \El(M)$ is  a $\ast$-isomorphism, and  $\Theta(xu_g)= \eta(g)w( \Psi(x)u_{\delta(g)}) w^*$ for all $x\in \El(N)$.
	Finally, proceeding as in the proof of \cite[Theorem 5.1]{CDK19} one can further show that the isomorphism $\Psi$ arises from a tensor of $\ast$-isomorphisms $\Phi_i: \El(N_i)\rar \El(M_i)$. We leave these details to the reader.\end{proof}

\begin{cor}\label{fg1} For any $G=N\rtimes Q \in \mathscr S $ the fundamental group of $\El(G)$ is trivial, i.e.\ $\mathcal F(\El(G))=1$. 
\end{cor}

\section{Fundamental Group of Factors Arising from Class $\mathscr V$}
In this section we describe another class of examples of property (T) factors with trivial fundamental group, namely the $\El(G)$ associated with the group in class $G \in \mathscr V$ from subsection \ref{Valettegr}. Using the properties highlighted there in combination with Popa--Vaes's Cartan rigidity results \cite{PV12} and Gaboriau's $\ell^2$-Betti numbers invariants \cite{Ga02} we show these are pairwise stably non-isomorphic property (T) factors with trivial fundamental group.

\begin{theorem}\label{fgv}
	Let $G \in \mathscr V$. 	The following properties hold: 
	\begin{itemize}
		\item [$(\rm i)$]  For every $G\in \mathscr V$ the fundamental group satisfies $\mathcal F(\El(G))=\{1\}$;
		\item [$(\rm ii)$] The family $\{\El(G) \,:\,G \in \mathscr V\}$ consists of pairwise stably non-isomorphic II$_1$ factors.
	\end{itemize}
\end{theorem}

\begin{proof}
	$(\rm i)$ Since $G \in \mathscr V$, then $G = G_{n_1} \times... \times G_{n_k}$, with $n_i \geq 2$ and $k\in \mathbb N$. Recall that for every $n \geq 2$,  $G_n =  \mathbb Z^{4(n + 1)} \rtimes \La_n \in \mathscr V$. First we show our statement for $k=1$, i.e. $\emm_n=\El(G_n)$ has trivial fundamental group. 
	To this end, let $\mathcal R_n$ be the orbit equivalence relation induced by the essentially free, ergodic probability measure preserving action $\Lambda_n \curvearrowright \mathbb T^{4(n + 1)}$. Thus $\mathcal L(\mathcal R_n) = \emm_n$ and \cite[Theorem 1.4]{PV12} implies that $\mathcal F(\emm_n) = \mathcal F(\mathcal R_n)$. Using Borel's result \cite{Bo83}, the $n$-th $\ell^2$-Betti number of $\Lambda_n$ is nonzero and finite. Therefore a combination of \cite[Corollaire 3.16]{Ga02} and \cite[Corollaire 5.7]{Ga02} implies that $\mathcal F(\mathcal R_n) = \{1\}$ and hence $\mathcal F(\emm_n) = \{1\}$. 
	\vskip 0.03 in
	The case $k\geq 2$ follows in a similar manner. Indeed let $m= \sum_{i=1}^k n_i$. Using Kunneth formula for $\ell^2$-Betti numbers, we see the $m$-th $\ell^2$-Betti number of $\La_{n_1} \times \cdots \times \La_{n_k}$ is nonzero and finite. Thus using \cite[Theorem 1.6]{PV12} and arguing exactly as in the previous paragraph, we get $\mathcal F(\El(G))= \{1\}$. 
	\vskip 0.07 in
	$(\rm ii)$  Let $G, H \in \mathscr V$ and $t>0$ such that $\mathcal L(G)\cong \El(H)^t$. Notice that $G = G_{n_1} \times ... \times G_{n_k}$ and $H= G_{m_i} \times ... \times G_{m_l}$, with $n_i,m_j \geq 2$. Denote by $\mathcal R_{\La_{n_1}\times ... \times \La_{n_k}}$ and $\mathcal R_{\La_{m_1}\times ... \times \La_{m_k}}$ the equivalence relations arising from the product actions $\times_i( \La_{n_i}\ca \mathbb Z^{4(n_i+1)})$ and $\times_i (\La_{m_i}\ca \mathbb Z^{4(m_i+1)})$, respectively.  Using \cite[Theorem 1.6]{PV12} we get the these equivalence relations are stably isomorphic, i.e. $\mathcal R_{\La_{n_1}\times ... \times \La_{n_k}}\cong (\mathcal R_{\La_{m_1}\times ...\times \La_{m_l}})^t$. Therefore using \cite[Theorem 1.16]{MS02} (see also \cite[Theorem ]{Dr19}) we have $k=l$ and after permuting the indices we have  $\mathcal R_{\La_{n_i}}\cong (\mathcal R_{\La_{m_i}})^{t_i}$ for some $t_1t_2...t_k=t$. However using \cite[Corollaire 0.4]{Ga02} (see also \cite{CZ88}) this further implies that $n_i = m_i$ and $t, t_1,t_2,...,t_k=1$; in particular, $G\cong H$.
\end{proof}

\noindent {\bf Remark.} We remark that we could have directly applied \cite[Theorem 4]{Va04} to the adjoint group of ${\rm Sp}(n, 1)$ in order to obtain examples of icc groups that satisfy the conclusion of the above theorem. Instead, we adapted the explicit and simpler construction given in \cite[Example 1, (a)]{Va04} to the case of ${\rm Sp}(n, 1)$.

\vskip 0.07in 
Our results  shed new light towards constructing non-isomorphic II$_1$ factors with property (T). While it is well known that there exist uncountably many pairwise non-isomorphic such factors, virtually nothing is known about producing explicit uncountable families. Indeed our Corollary \ref{fg1} and Theorem \ref{fgv} give such examples. 

\begin{cor} For any $G=N\rtimes Q \in \mathscr S$ or $G= G_{n_1}\times ...\times G_{n_k}\in \mathscr V$ then the set of all amplifications $\{\El(G)^t\,:\, t\in (0,\infty)\}$ consists of pairwise non-isomorphic II$_1$ factors with property (T). \end{cor}

We end this section with a unique prime factorization result of independent interest regarding groups in class $\mathscr V$. We notice this can also be employed to bypass the usage of \cite[Theorem 1.16]{MS02} in Theorem \ref{fgv}. For the proof we will use  techniques similar to the recent methods used to classify tensor product decompositions of II$_1$ factors, \cite{DHI16,CdSS17,Dr19}.

\begin{theorem} Assume that $G=G_{n_1}\times ...\times G_{n_k}\in \mathscr V$. Assume that $\emm=\El(G)=\mathcal P_1\bar\otimes\mathcal P_2$ for any $P_i$ s II$_1$ factors. Then one can find a partition $I_1\sqcup I_2=\{1,...,k\}$, a unitary $u\in \emm$ and positive scalars $t_1,t_2$ with $t_1t_2=1$ such that $\El(\times_{i\in I_1}G_{n_i})=uP_1^{t_1}u^*$ and $\El(\times_{i\in I_2}G_{n_i})=uP_2^{t_2}u^*$.
\end{theorem}

\begin{proof} Throughout this proof for every subset $F\subset \{1,...,k\}$ we denote its complement by $\overline F= \{1,...,k\}\setminus F$ and by $G_F=\times_{i\in F}G_{n_i}$ the sub-product group of $G$ supported on $F$. Using these notations we first prove the following \begin{claim}\label{intoneless} For every $I\subseteq \{1,...,k\}$ and $e \in\mathscr P(\El(G_I))$ assume that $\mathcal C,\mathcal D\subseteq e\El(G_I)e$ are two commuting diffuse property (T) von Neumann subalgebras.  Then for every $i\in I$ we have either $ \mathcal C\prec_\emm \El(G_{I\setminus \{i\}})$, or $\mathcal D\prec_\emm \El(G_{I\setminus \{i\}})$.
	\end{claim}
	
	\noindent \textit{Proof of Claim \ref{intoneless}}. Fix $\bee\subset \mathcal C$  an arbitrary diffuse amenable von Neumann subalgebra. Writing $\emm= (\El(G_{I \setminus \{i\}})\bar\otimes \Aa_i) \rtimes \La_i$ and using \cite[Theorem 1.4]{PV12} we have either \begin{enumerate}\item [a)] $\bee \prec_\emm\El(G_{I \setminus \{i\}})\bar\otimes \Aa_i$ or \item [b)] $\mathscr N_\emm (\bee )''$ is amenable relative to $\El(G_{I \setminus \{i\}})\bar\otimes \Aa_i$ inside $\emm$. 
	\end{enumerate}
	Since $\mathcal D\subseteq \mathscr N_\emm (\bee )''$ and $\mathcal D$ has property (T) then b) implies that  $\mathcal D \prec_\emm \El(G_{I \setminus \{i\}})\bar\otimes \Aa_i$. Also if a) holds for all such $\bee$'s then by \cite[Corollary F.14]{BO08} we have $\mathcal C \prec_\emm \El(G_{I \setminus \{i\}})\bar\otimes \Aa_i$.
	Next assume that $\mathcal C \prec_\emm \El(G_{\bar i})\bar\otimes \Aa_i$. Thus there exist projections $p\in \mathcal C , q\in\El(G_{I \setminus \{i\}})\bar\otimes \Aa_i$  a partial isometry $w\in \emm$ and a $\ast$-isomorphism on its image $\phi: p\mathcal C p \rar \mathcal Q: =\phi(p\mathcal C p)\subseteq q (\El(G_{I \setminus \{i\}})\bar\otimes \Aa_i )q$ such that \begin{equation}\label{inteq2}\phi(x)w=wx\text{  for all }x\in p\mathcal C p.\end{equation} We also have that $w^*w\in \mathcal D p$ and $ww^*\in \mathcal Q'\cap q\emm q$ and we can arrange that the support satisfies ${\rm sup} (E_{\El(G_{I \setminus \{i\}})\bar\otimes \Aa_i }(ww^*))=q$. Since $\mathcal C$ has property (T) then so does $p\mathcal C p$ and also $Q$. Since $Q\subseteq \El(G_ {I \setminus \{i\}})\bar \otimes \mathcal A_i$ and $\mathcal A_i$ is amenable then we have that $\mathcal Q\prec_{\El(G_{I \setminus \{i\}})\bar \otimes \mathcal A_i} \El(G_{I \setminus \{i\}})$. Therefore one can find projections $r\in \mathcal Q , t\in\El(G_{I \setminus \{i\}})$  a partial isometry $v\in \El(G_{I \setminus \{i\}})\bar \otimes \mathcal A_i$ and a $\ast$-isomorphism on its image $\psi: r\mathcal Q r \rar  t \El(G_{I \setminus \{i\}})t$ such that 
	\begin{equation}\label{inteq1}\psi(x)v=vx\text{  for all }x\in r\mathcal Q r.\end{equation}
	Letting $s:= \phi^{-1}(r)\in \mathcal C$ the equations \eqref{inteq2}-\eqref{inteq1} show that for every $y\in s\mathcal C s$ we have  $ \psi(\phi(y))vw = v\phi(y)w= vw y$. Moreover, using ${\rm sup} (E_{\El(G_{I \setminus \{i\}})\bar\otimes \Aa_i }(ww^*))=q$ and $v\in \El(G_{I \setminus \{i\}})\bar\otimes \Aa_i$ a simple calculation shows that $vw\neq 0$. Altogether, these imply that $\mathcal C \prec_\emm \mathcal L(G_{I \setminus \{i\}}) $.
	\vskip 0.04in
	
	In a similar fashion,  $\mathcal D \prec_\emm \mathcal L(G_{I \setminus \{i\}})\bar \otimes \mathcal A_i $ implies $\mathcal D \prec_\emm \mathcal L(G_{I \setminus \{i\}}) $, and the claim obtains.$\hfill\blacksquare$
	\vskip 0.07in 
	In the remaining part we derive the conclusion of the theorem. Since the subgroups $G_{\bar {i}}$ are normal in $G$ then using \cite[Lemma 2.6]{DHI16} and the Claim \ref{intoneless} (for $\mathcal C=\mathcal P_1$ and $\mathcal D=\mathcal P_2$) inductively one can find nonempty minimal subsets $I_1,I_2\subsetneq \{1,...,k\}$ such that \begin{equation}\label{twosint}\mathcal P_1\prec^s_\emm \El(G_{I_1}) \text{ and  }\mathcal P_2\prec^s_\emm \El(G_{I_2}).\end{equation} 
	As $\mathcal P_1 \vee \mathcal P_2=\emm$ and $\mathcal P_1\prec_\emm \El(G_{I_1})$ then  using \cite[Lemma 4.2]{CdSS17} one can find projections $p \in \mathcal P_1 , f \in  L(G_{I_1})$, a partial isometry $w \in  \emm$, and a unital injective $\ast$-homomorphism $\theta : p\mathcal P_1p \rar \mathcal R:= \theta(pPp) \subseteq f \El(G_{I_1})f$ such that\begin{enumerate}
		\item [a)] $\theta(x)w=wx$ forall $x\in p\mathcal P_1p$;
		\item [b)] ${\rm sup}(\El( G_{I_1})(ww^*)) = f$;
		\item [c)]  $\mathcal R \vee (\mathcal R' \cap f \El(G_{I_1} )f) \subseteq f\El(G_{I_1})f$ is a finite index irreducible inclusion of II$_1$ factors. 
	\end{enumerate}
	
	Since $\mathcal L(G_{I_1})$ has property (T) then so is $f\mathcal L(G_{I_1})f$. Using the finite index condition in part c) above we conclude that $\mathcal R \vee (\mathcal R' \cap f \El(G_{I_1} )f)$ has property (T); in particular $\mathcal R$ and  $\mathcal R' \cap f \El(G_{I_1} )f$ are commuting property (T) subfactors of $ f \El(G_{I_1})f$. 
	
	Next we claim that $\mathcal T:=\mathcal R' \cap f \El(G_{I_1} )f$ is finite dimensional. Assume by contradiction it is not. Thus $\mathcal T$ is a diffuse property (T) factor. Note that $\mathcal N_{f \El(G_{I_1} )f}(\mathcal T)' \cap f\El(G_{I_1})f \subseteq \mathcal (R \vee (\mathcal R' \cap f \El(G_{I_1} )f))' \cap f \El(G_{I_1} )f= \mc f $. By Claim \ref{intoneless}, and \cite[Lemma 2.4(2)]{DHI16}, we get that for every $i\in I$ we have either i) $ \mathcal R\prec^s_\emm \El(G_{I\setminus \{i\}})$, or ii) $\mathcal T\prec^s_\emm \El(G_{I\setminus \{i\}})$. If possibility ii) would hold for all $i \in I_1$, since $G_{I_1\setminus \{i\}}$
	is normal in $\mathcal G_{I_1}$ then by \cite[Lemma 2.6]{DHI16} we would have  $\mathcal T \prec_{\mathcal L(G_{I_1})} \cap_{i\in I_1} \mathcal L (G_{I_1\setminus \{i\}})=\mathbb C1$, which is a contradiction. Hence there must be $j\in I_1$ such that $\mathcal R \prec^s \mathcal L(G_{I_1\setminus \{j\}})$. However, by \cite[Lemma 3.1]{Va07} this would further imply that $\mathcal P_1 \prec^s \mathcal L(G_{I_1 \setminus \{j\}})$ which contradicts the minimality of $I_1$. Thus our assumption that $\mathcal T$ is diffuse is false. 
	
	Now since $\mathcal T$ is finite dimensional condition c) gives in particular that $\mathcal R \subseteq f \El(G_{I_1})f$ is finite index and hence $\mathcal L(G_{I_1})\prec_\emm \mathcal P_1$ and since $L(G_{I_1})$ has property (T) we conclude that $\mathcal L(G_{I_1})\prec^s_\emm \mathcal P_1$. Proceeding in a similar fashion we also get $\mathcal L(G_{I_2})\prec^s_\emm \mathcal P_2$. Altogether, these show that $I_1\sqcup I_2=\{1,...,n\}$ is a proper partition. Therefore using these intertwinings in combination with \eqref{twosint} by \cite[Theorem 6.1]{DHI16} one can find a product decomposition $G= \G_1\times \G_2$ such that $\G_i$ is commensurable with $G_{I_i}$  for all $i=1,2$ and there exist a unitary $u\in \emm$ and scalars $t_1t_2=1$ such that $\El(\G_1) =u(P_1^{t_1})u^*$ and $\El(\G_2)= u( P_2^{t_2})u^*$. Finally, since for every $i=1,2$ the group $\G_i$  is commensurable to $G_{I_i}$ and $G$ is icc, torsion free one can check that in fact $\G_i = G_{I_i}$ and the desired conclusion follows.\end{proof}

An alternative proof of the above theorem can be given by using the notion of spatially commensurable von Neumann algebras \cite[Definition 4.1]{CdSS17} together with the results  \cite[Lemma 4.2, Theorems 4.6-4.7]{CdSS17}. This uses essentially the same arguments as before and  bypasses the usage of \cite[Theorem 6.1]{DHI16}.  We leave the details to the reader. 

\section*{Acknowledgments}
The authors would like to thank Adrian Ioana, Jesse Peterson and Stefaan Vaes for many helpful comments and suggestions regarding this paper.
Part of this work was completed while the fourth author was visiting the University of Iowa. He is grateful to the mathematics department there for their hospitality.

\noindent
\textsc{Department of Mathematics, The University of Iowa, 14 MacLean Hall, Iowa City, IA 52242, U.S.A.}\\
\email {ionut-chifan@uiowa.edu} \\
\email{sayan-das@uiowa.edu}\\
\textsc{Universit\'e Paris-Saclay, CNRS, Laboratoire de math\'ematiques d'Orsay, 91405, Orsay, Institut Universitaire de France, FRANCE}\\
\textsc{Graduate School of Mathematical Sciences, The University of Tokyo, Komaba, Tokyo, 153-8914, JAPAN}\\
\email{cyril.houdayer@universite-paris-saclay.fr}\\
\textsc{Department of Mathematics, Vanderbilt University, 1326 Stevenson Center, Nashville, TN 37240, U.S.A.}\\
\email{krishnendu.khan@vanderbilt.edu}

\end{document}